\documentclass[11pt]{article}

\usepackage[pdftex]{color}
\usepackage[pdftex]{graphicx}
\usepackage[english]{babel}

\usepackage[latin1]{inputenc} % per accents, pero no tots :(
\usepackage[letterpaper]{geometry}
\usepackage{authblk}
\usepackage[bookmarks=true,colorlinks=true,linkcolor=black,citecolor=black,menucolor=black,urlcolor=black]{hyperref}

\usepackage{enumerate}
\usepackage[dvipsnames]{xcolor}

\usepackage{mathrsfs} % per fer una lletra especial per la sigma-Algebra

\def\1{{\rm l}\hskip -0.21truecm 1}

\usepackage[english]{babel}
\usepackage{amsmath}
\usepackage{amsfonts}
\usepackage{amssymb}
\usepackage{amsthm}
\usepackage{enumerate}
\usepackage{mathrsfs} % sigma-algebra \mathscr{F}
\usepackage{chngcntr}
\usepackage{float}
\usepackage{graphicx}
\usepackage{listings}

\newcommand{\N}{\mathbb{N}}
\newcommand{\R}{\mathbb{R}}

\newtheorem{thm}{Theorem}[section]
\newtheorem{prop}[thm]{Proposition}
\newtheorem{lemma}[thm]{Lemma}

\newtheorem{definition}[thm]{Definition}
\newtheorem{rem}[thm]{Remark}

\def\1{{\rm l}\hskip -0.21truecm 1}

\begin{document}

\begin{center}
\LARGE{Convergence of delay equations driven by a H\"older continuous function of order $\beta\in(\frac13,\frac12)$}
\end{center}

\begin{center}
\begin{tabular}{@{}c@{}}
Mireia Besal\'u\textsuperscript{1} \\
\small mbesalu@ub.edu
\end{tabular}%
\quad \quad
\begin{tabular}{@{}c@{}}
Giulia Binotto\textsuperscript{1} \\
\small gbinotto@ub.edu
\end{tabular}
\quad \quad
\begin{tabular}{@{}c@{}}
Carles Rovira\textsuperscript{1} \\
\small carles.rovira@ub.edu
\end{tabular}
\smallskip

Departament de  Matem\`atiques i Inform\`atica,

 Universitat de Barcelona, Barcelona.
\end{center}

\let\thefootnote\relax\footnotetext{
\textsuperscript{1}
The authors have been supported by the grant MTM2015-65092-P from MINECO, Spain}

\vskip 3pt
\begin{abstract}%
In this note we show that, when the delay goes to zero, the solution of multidimensional delay differential  equations driven by a H\"older continuous function of order $\beta \in (\frac13,\frac12)$ converges with the supremum norm to the solution for the equation without delay.  As an application, we discuss the applications to  stochastic differential equations.
\end{abstract}

\vskip 6pt

\noindent {\bf AMS 2000 MSC:} 60H05, 60H07
\vskip 6pt

\noindent {\bf Running head:}  Convergence of delay equations

\section{Introduction} \label{dSDE_Intro}

In \cite{HN} Hu and Nualart establish using fractional calculus  the existence and uniqueness of solution for the dynamical system
	\[ dx_t = f(x_t) \,dy_t, \]
where $y$ is a H\"older continuous function of order $\beta\in(\frac13,\frac12)$.  In this work they give an  explicit expression for  the integral $\int_{0}^{t}f(x_s)dy_s$
that depends on the functions $x$, $y$ and a quadratic multiplicative functional 
$x\otimes y$.  As an example of a path-wise approach to classical stochastic calculus, they apply these results to solve stochastic differential equations driven by a multidimensional Brownian motion.  Using the same approach, Besal\'u and Nualart \cite{BN} got estimates
for the supremum norm of the solution and Besal\'u, M\'arquez and Rovira \cite{BMR} studied delay equations with non-negativity constraints.

The work of  Hu and Nualart   \cite{HN}  is an extension of the previous paper of 
 Nualart and R\u{a}\c{s}canu  \cite{NR} where they  study the  dynamical
systems   $dx_{t}=f(x_{t})dy_{t}$, where the control function $y$ is H\"{o}lder
continuous of order $\beta >\frac{1}{2}$.
In this case the Riemann-Stieltjes integral $\int_{0}^{t}f(x_s)dy_s$     can be expressed as a Lebesgue integral
using fractional derivatives  following the ideas of Z\"{a}hle \cite{Z}.

All this papers have to be seen in the framework of the theory of rough path analysis and the path-wise approach to classical stochastic calculus. This theory have  been developed from the initial paper
by Lyons \ \cite{L} and has generated a wide literature (see, for instance,
Lyons and Qian \cite{L-Q},  Friz and Victoir  \cite{FV2}, Lejay \cite{Le} or  Gubinelli  \cite{Gu}).  
On the other hand, we
refer  for instance to Coutin and Lejay \cite{CA}, Friz and Victoir \cite{Fr}, Friz \cite{FV} and Ledoux {\it et al.} \cite{LQZ} for some
applications of   rough path analysis to the stochastic calculus.

\smallskip

Delay differential equations rise from the need to study models that behave more like the real processes. They find their applications in dynamical systems with aftereffects or when the dynamics are subjected to propagation delay. Some examples are epidemiological models with incubation periods that postpone the transmission of disease, or neuronal models where the spatial distribution of neurons can cause a delay in the transmission of the impulse.
Sometimes the delay avoids some usual problems, but in general it adds difficulties and cumbersome notations.

\smallskip

The purpose of our paper is to consider the following differential equation with delay:
	\begin{eqnarray*} 
		x^r_t &=& \eta_0 + \int_0^t b(u,x^r) \,du + \int_0^t \sigma(x^r_{u-r}) \,dy_u, \qquad t\in(0,T], \\
		x^r_t &=& \eta_t, \qquad t\in[-r,0],
	\end{eqnarray*}
where $r$ denotes a strictly positive time delay and
where $\eta:[-r,0]\to\R^d$ is a smooth function and $y$ is a H\"older continuous function of order $\beta\in(\frac13,\frac12)$  and the hereditary term $b(u,x)$ depends on the path $\{x_s, \,0\leq s\leq u\}$.  From Hu and Nualart \cite{HN} and Besal\'u, M\'arquez and Rovira \cite{BMR} it is easy to check that there exists a unique solution of this equation. Our aim is to prove that it converges almost surely in the supremum norm to the solution of the differential equation without delay
	\begin{equation*}
		x_t = \eta_0 + \int_0^t b(u,x_u) \,du + \int_0^t \sigma(x_u) \,dy_u, \qquad t\in[0,T],
	\end{equation*}
when the delay tends to zero. Our approach is based on the techniques of the classical fractional calculus and it is inspired by \cite{HN}.
Finally we will apply these results to stochastic differential equations driven by Brownian motion.

\smallskip

The case when $\beta>\frac12$ is studied by Ferrante and Rovira in \cite{FR}. They prove that the solution of the delay equation converges, almost surely and in $L^p$, to the solution of the equation without delay and then apply the result pathwise to fractional Brownian motion with Hurst parameter $H>\frac12$.

\smallskip

With a different approach based on a slight variation of the Young integration theory, called algebraic integration, Le\'on and Tindel prove in \cite{LT} the existence of a unique solution for a general class of delay differential equations driven by a H\"older continuous function with parameter greater that $\frac12$. They obtain some estimates of the solution which allow to show that the solution of a delay differential equation driven by a fractional Brownian motion with Hurst parameter $H>\frac12$ has a $\mathcal{C}^\infty$-density.

In the case when $\beta<\frac12$ more difficulties appear and in literature.
In \cite{NNT}, Neuenkirch, Nourdin and Tindel consider delay differential equation driven by a $\beta$-H\"older continuous function with $\beta>\frac13$.  The authors show the existence of a unique solution for these equations under suitable hypothesis. Then, they apply these results to a delay differential equations driven by a fractional Brownian motion with Hurst parameter $H>\frac13$.
These results are extended by Tindel and Torrecilla  in \cite{TT} to the deterministic case of order $\beta>\frac14$ and the corresponding stochastic case with Hurst parameter $H>\frac14$.

\smallskip

The  paper is organized as follows. The following section is devoted to introduce some notation. In Section \ref{dSDE_Main} we define the equations and the solutions we work with and we describe our main result. Section \ref{dSDE_EstimatesInt} contains technical estimates for the study of the integrals. In Section \ref{dSDE_EstimatesSol}  we give some estimates of the solutions of our equations. In Section  \ref{dSDE_Proof} we give the proof of the main theorem. Finally, the last section is dedicated to give an exemple of the application of the main theorem studying stochastic differential equations driven by Brownian motion.

\section{Preliminaries} \label{dSDE_Prel}

First we recall some definitions and results presented in Hu and Nualart \cite{HN}.
\smallskip

Fix a time interval $[0,T]$. For any function $x:[0,T]\rightarrow \R^d$, the $\beta$-H\"{o}lder norm of $x$ on the interval $[s,t] \subset \lbrack 0,T]$, where $0<\beta\leq1$, will be denoted by
\begin{equation*}
\left\| x\right\| _{\beta(s,t)}=\sup_{s<u<v<t}\frac{|x_v-x_u|}{(v-u)^\beta}.
\end{equation*}
If $\Delta_T :=\left\{ (s,t):0\leq s<t\leq T\right\} $, for any $(s,t)\in\Delta_T$ and for any $g:\Delta_T\rightarrow\R^d$ we set 
\begin{equation*}
\left\|g\right\|_{\beta(s,t)}=\sup_{s<u<v<t}\;\frac{|g_{u,v}|}{(v-u)^{\beta}}.
\end{equation*}
Moreover, $\left\| \cdot \right\|_{\infty(s,t)}$ will denote the supremum norm in the interval $(s,t)$. 

\smallskip

Hu and Nualart \cite{HN} prove an explicit formula for integrals of the form $\int_a^b f(x_u) \,dy_u$ in terms of $x$, $y$ and $x\otimes y$ and  transform the dynamical system $dx_t=f(x_t)\,dy_t$ into a closed system of equations involving only $x$, $x\otimes y$ and $x\otimes (y\otimes y)$ . Fix $0<\beta\leq1$. From Lyons \cite{L} we need  to  introduce the  definition of  $x\otimes y$:
\begin{definition}\label{def:otimes}
We say that $(x,y,x\otimes y)$ is an {$(d,m)$-dimensional $\beta $-H\"{o}lder continuous multiplicative functional} if:
\begin{enumerate}
\item $x:[0,T]\rightarrow \mathbb{R}{}^{d}$ and $y:[0,T]\rightarrow \R^m$ are $\beta $-H\"{o}lder continuous functions,
\item $x\otimes y:\Delta_T\rightarrow\R^d\otimes\R^m$ is a continuous function satisfying the following properties:
\begin{enumerate} 
\item {\upshape(Multiplicative property)} For all $s\leq u\leq t$ we have 
\begin{equation*}
(x\otimes y)_{s,u}+(x\otimes y)_{u,t}+(x_u-x_s)\otimes (y_t-y_u)=(x\otimes y)_{s,t}.
\end{equation*}
\item For all $(s,t)\in\Delta_T$ 
\begin{equation*}
|(x\otimes y)_{s,t}| \leq c|t-s|^{2\beta}.
\end{equation*}
\end{enumerate}
\end{enumerate}
\end{definition}
      
We  denote by $M_{d,m}^{\beta}(0,T)$ the space of $(d,m)$-dimensional $\beta$-H\"{o}lder continuous multiplicative functionals. Furthermore, we will denote by $M_{d,m}^{\beta}(a,b)$ the obvious extension of the definition $M_{d,m}^{\beta}(0,T)$ to a general interval $(a,b)$.
Let us recall the following functional on $M_{d,m}^{\beta}(0,T)$ for $a,b\in\Delta_T$:

	\begin{equation} \label{Phi2}
		\Phi_{\beta(a,b)}(x,y) = \|x \otimes y \|_{2\beta(a,b)} + \|x\|_{\beta(a,b)} \|y\|_{\beta(a,b)}.
	\end{equation}
Moreover, if $(x,y,x\otimes y)$ and $(y,z,y\otimes z)$ belongs to $M_{d,m}^{\beta}(0,T)$ we define
	\begin{eqnarray} \label{Phi3}
		\Phi_{\beta(a,b)}(x,y,z) &=& \|x\|_{\beta(a,b)} \|y\|_{\beta(a,b)} \|z\|_{\beta(a,b)} + \|z\|_{\beta(a,b)} \|x\otimes y\|_{2\beta(a,b)} \nonumber \\
		&& \hspace{10mm} + \|x\|_{\beta(a,b)} \|y\otimes z\|_{2\beta(a,b)}.
	\end{eqnarray}
From these definitions it follows that
	\begin{equation} \label{eq_Phi2}
		\big\| (x\otimes y)_{\cdot,b} \big\|_{\beta(a,b)} \leq \Phi_{\beta(a,b)}(x,y) (b-a)^{\beta}
	\end{equation}
and
	\begin{equation} \label{eq_Phi3}
		\big\| x\otimes(y\otimes z)_{\cdot,b} \big\|_{2\beta(a,b)} \leq K \,\Phi_{\beta(a,b)}(x,y,z) (b-a)^{\beta}
	\end{equation}
that are equations (3.29) and (3.30) of \cite{HN} respectively.    
We refer to \cite{HN} and \cite{L} for a more detailed presentation on $\beta $-H\"{o}lder continuous multiplicative functionals.

\smallskip

To define the integral $\int_a^b f(x_u) \,dy_u$  we use the construction of the integral given by Hu and Nualart in \cite{HN}. They are inspired by the work of Z\"alhe \cite{Z} and use fractional derivatives. We refer to Hu and Nualart in \cite{HN} for the details.

\smallskip
In the sequel, $K$ denotes a generic constant that may depend on the parameters $\beta,\alpha$, $\lambda$ and $T$ and vary from line to line.

\section{Main result} \label{dSDE_Main}

Consider the following differential equation on $\R^d$ with delay:
	\begin{eqnarray} \label{eq:delay}
		x^r_t &=& \eta_0 + \int_0^t b(u,x^r_u) \,du + \int_0^t \sigma(x^r_{u-r}) \,dy_u, \qquad t\in[0,T], \nonumber \\
		x^r_t &=& \eta_t, \qquad t\in[-r,0),
	\end{eqnarray}
where $x$ and $y$ are H\"older continuous functions of order $\beta\in(\frac13,\frac12)$, $\eta$ is a continuous function and $r$ denotes a strictly positive time delay.

\smallskip

Set the following hypothesis:
	\begin{itemize}
	    \item[\bfseries(H1)] \label{H1} $\sigma:\R^d\rightarrow\R^d\times\R^m$ is a bounded and continuously twice differentiable function such that $\sigma'$ and $\sigma''$ are bounded and $\lambda$-H\"older continuous for $\lambda>\frac{1}{\beta}-2$.
	    \item[\bfseries(H2)] \label{H2}
	    $b:[0,T]\times \R^d\rightarrow \R^d$ is a measurable function such that there exists $b_0\in L^\rho(0,T;\R^d)$ with $\rho\geq 2$ and $\forall N\geq 0$ there exists $L_N>0$ such that:
    		\begin{itemize}
		        \item[{\bf (1)}] $\left|b(t,x_t)-b(t,y_t)\right|\leq L_N |x_t-y_t|,\; \forall x,y$ such that $|x_t|\leq N$, $|y_t|\leq N$ $\forall t\in[0,T]$,
		        \item[{\bf (2)}] $\left|b(t,x_t)\right|\leq L_0 |x_t|+b_0(t),\quad \forall t\in[0,T]$.
		    \end{itemize}
	    \item[\bfseries(H3)] \label{H3} $\sigma$ and $b$ are bounded functions.
	\end{itemize}

Conditions {\bfseries(H1)} and {\bfseries(H2)} are a particular case of the hypothesis for the proof of existence and uniqueness of solution of the delay equation \eqref{eq:delay}, while condition {\bfseries(H3)} is necessary to prove that the solution is bounded.

\smallskip
We denote by $(x,y,x\otimes y)\in M^\beta_{d,m}(0,T)$ the solution of the stochastic differential equation on $\R^d$ without delay:
	\begin{equation} \label{eq:nodelay}
		x_t = \eta_0 + \int_0^t b(u,x_u) \,du + \int_0^t \sigma(x_u) \,dy_u, \qquad t\in[0,T].
	\end{equation}
In \cite{HN}, Hu and Nualart prove under the assumptions that $\sigma:\R^d\rightarrow\R^d\times\R^m$ is a continuously differentiable function such that $\sigma'$ is $\lambda$-H\"older continuous, where $\lambda>\frac1\beta-2$, $\sigma$ and $\sigma'$ are bounded, and $(y,y,y\otimes y)\in M^\beta_{m,m}(0,T)$, that there exists a bounded solution $(x,y,x\otimes y)\in M^\beta_{d,m}(0,T)$ for the differential equation (\ref{eq:nodelay}) with $b \equiv 0$.
Moreover, if $\sigma$ is twice continuously differentiable with bounded derivatives and $\sigma''$ is $\lambda$-H\"older continuous, where $\lambda>\frac1\beta-2$, the solution is unique.
Here the authors consider the equation without the hereditary term, but the results can be easily extended to the case when the hereditary term does not vanish.
 If $(y,y,y\otimes y)\in M_{m,m}^\beta(0,T)$, then  consider  $(x,y,x\otimes y)$, where
	\begin{equation} \label{eq:nodelay2}
		(x\otimes y)_{s,t} = \int_s^t (y_t-y_u) b(u,x_u) \,du + \int_s^t \sigma(x_{u}) \,d_u(y\otimes y)_{\cdot,t}.
	\end{equation}
Then  a solution of equation \eqref{eq:nodelay} is an element of $M_{d,m}^\beta(0,T)$ such that  \eqref{eq:nodelay} and \eqref{eq:nodelay2} hold.

\smallskip

On the other hand, following the ideas contained in \cite{BMR}, it is easy to show that there exists a unique solution $(x^r,y,x^r\otimes y)\in M^\beta_{d,m}(0,T)$ of the delay equation \eqref{eq:delay} . This is proved assuming that $\sigma$ and $b$ satisfy the hypothesis {\bfseries (H1)} and {\bfseries (H2)}, respectively, with $\rho\geq\frac{1}{1-\beta}$, $(\eta_{\cdot-r},y,\eta_{\cdot-r}\otimes y)\in M^\beta_{d,m}(0,r)$ and $(y_{\cdot-r},y,y_{\cdot-r}\otimes y)\in M^\beta_{m,m}(r,T)$.
Assuming also that hypothesis {\bfseries (H3)} is  satisfied, we obtain that the solution is bounded. In this case, 
 $(x^r,y,x^r\otimes y)\in M_{d,m}^\beta(-r,T)$ is the solution, where $(x^r\otimes y)_{s,t}$ is defined as follows:
	\begin{itemize}
		\item for $s<t\in[-r,0)$: \qquad
	        	$(x^r\otimes y)_{s,t} = (\eta\otimes y)_{s,t} = \int_s^t (y_t-y_u) \,d\eta_u,$	    
	    \item for $s\in[-r,0)$ and $t\in[0,T]$,
	        \begin{eqnarray} \label{eq:delayxy2}
		        (x^r\otimes y)_{s,t} &=& (\eta\otimes y)_{s,0} + \int_0^t (y_t-y_u) b(u,x^r_u) \,du \nonumber \\
		        && \hspace{5mm} + \int_0^t \sigma(x^r_{u-r}) \,d_u(y\otimes y)_{\cdot,t} + (\eta_0-\eta_s)\otimes(y_t-y_0), \nonumber\\
	        \end{eqnarray}
	    \item for $s<t\in[0,T]$:\qquad
	        $	        	(x^r\otimes y)_{s,t} = \int_s^t (y_t-y_u) b(u,x^r_u) \,du + \int_s^t \sigma(x^r_{u-r}) \,d_u(y\otimes y)_{\cdot,t}.$
	\end{itemize}

\smallskip

Let $\beta\in(\frac13,\frac12)$ and set $\beta'=\beta-\varepsilon$, where $\varepsilon>0$ is such that $\beta-2\varepsilon>0$ and $\lambda>\frac{1}{\beta-\varepsilon}-2$. Set $r_0 \in (0,T)$.
The main result of the paper is the following theorem:
	\begin{thm} \label{thm_determ}
		Suppose that $(y,y,y\otimes y)$ belongs to $M^\beta_{d,m}(0,T)$ and $(y_{\cdot-{r}},y,y_{\cdot-r}\otimes y)$ belongs to $M^\beta_{d,m}(r,T)$ for all $0<r \le r_0$. Assume that $\sigma$ and $b$ satisfy {\upshape\bfseries (H1)} and {\upshape\bfseries (H2)} respectively, and both satisfy {\upshape\bfseries (H3)}. Assume also that $(\eta_{\cdot-r_0},y,\eta_{\cdot-r_0}\otimes y)\in M^\beta_{d,m}(0,r_0)$,  $\|\eta\|_{\beta(-r_0,0)}<\infty$ and $\sup_{r\leq r_0} \Phi_{\beta(0,r)}(\eta_{\cdot-r},y)<\infty$ and suppose that $\|(y-y_{\cdot-r})\otimes y\|_{2\beta'(r,T)} \rightarrow 0$ and $\|y_{\cdot-r}\otimes (y-y_{\cdot-r})\|_{2\beta'(r,T)} \rightarrow 0$ when $r$ tends to zero. Then,
			\[ \lim_{r\rightarrow0} \|x-x^r\|_{\infty} = 0 \quad  \qquad {\rm and} \qquad 
			 \lim_{r\rightarrow0} \|(x\otimes y)-(x^r\otimes y)\|_{\infty} = 0 \quad  \]
	\end{thm}

\section{Estimates of the integrals} \label{dSDE_EstimatesInt}

In this section we will give some estimates for the integrals appearing in our equations. We begin  recalling Propositions 3.4 and Proposition 3.9 from \cite{HN}.

	\begin{prop}\label{prop1HN}
		Let $(x,y,x \otimes y)$ be in $M_{d,m}^\beta(0,T)$. Assume that $f:\R^d\longrightarrow\R^m$ is a continuous differentiable function such that $f'$ is bounded and $\lambda$-H\"older continuous, where $\lambda>\frac{1}{\beta}-2$. Then, for any $0\leq a<b\leq T$, we have
		\begin{eqnarray*}
			\Big| \int_a^b f(x_u)\,dy_u \Big| &\leq& K |f(x_a)| \,\|y\|_{\beta(a,b)} (b-a)^\beta + K\,\Phi_{\beta(a,b)}(x,y) \\
			&& \hspace{5mm} \times \left( \|f'\|_\infty + \|f'\|_\lambda \|x\|^\lambda_{\beta(a,b)} (b-a)^{\lambda\beta} \right) (b-a)^{2\beta},
		\end{eqnarray*}
		where $\Phi_{\beta(a,b)}(x,y)$ is defined in (\ref{Phi2}).
	\end{prop}

	\begin{prop}\label{prop2HN}
		Suppose that $(x,y,x\otimes y)$ and $(y,z,y\otimes z)$ belong to $M_{d,m}^\beta(0,T)$. Let $f:\R^d\longrightarrow\R^m$ be a continuously differentiable function such that $f'$ is $\lambda$-H\"older continuous and bounded, where $\lambda>\frac{1}{\beta}-2$.
		Then, the following estimate holds:
		\begin{eqnarray*}
			&& \Big| \int_a^b f(x_u) \,d_u(y\otimes z)_{\cdot,b} \Big| \leq K \,|f(x_a)| \,\Phi_{\beta(a,b)}(y,z) (b-a)^{2\beta} \\
			&& \hspace{20mm} + \,K \,\left( \|f'\|_\infty + \|f'\|_\lambda \|x\|^\lambda_{\beta(a,b)} (b-a)^{\lambda\beta} \right) \Phi_{\beta(a,b)}(x,y,z) (b-a)^{3\beta},
		\end{eqnarray*}
		where $\Phi_{\beta(a,b)}(x,y,z)$ is defined in (\ref{Phi3}).
	\end{prop}

The following propositions give some estimates useful for the proof of Theorem \ref{thm_determ}. \\
First we give a result for a function $b$ that fulfills conditions {\upshape\bfseries (H2)}.  
	\begin{prop} \label{prop:normbetab}
    	Assume that $b$ satisfies {\upshape\bfseries (H2)}. Let $x,\widetilde{x}\in C(0,T;\R^d)$ such that $\|x\|_{\infty}\leq N$ and $\|\widetilde{x}\|_{\infty}\leq N$. Then, for $0\leq a<b\leq T$,
    		\[ \bigg| \int_a^b \big[ b(u,x_u)-b(u,\widetilde{x}_u) \big] \,du \bigg| \leq L_N (b-a) \|x-\widetilde{x}\|_{\infty(a,b)}. \]
	\end{prop}
\begin{proof}
	It follows easily using the Lipschitz property of hypothesis {\bfseries(H2)}.
\end{proof}
    
In order to give some results for a function $f$ under conditions {\upshape\bfseries (H1)} we need to introduce some notation:
\begin{eqnarray*}
				G^1_{\beta(a,b)}(f,x,\widetilde{x},y) &=& K \Big[ \|y\|_{\beta} \,\|f'\|_\infty + \big( \|f''\|_\infty + \|f''\|_\lambda (\|x\|^\lambda_{\beta(a,b)} + \|\widetilde{x}\|^\lambda_{\beta(a,b)}) (b-a)^{\lambda\beta} \big) \\
				&& \hspace{10mm} \times \big( \Phi_{\beta(a,b)}(x,y) + \|y\|_{\beta} \,\|\widetilde{x}\|_{\beta(a,b)} \big) \Big], \\
				G^2_{\beta(a,b)}(f,x,\widetilde{x},y) &=& K \Big[ \|y\|_{\beta} \,\|f'\|_\infty + \|f''\|_\infty \big( \Phi_{\beta(a,b)}(x,y) + \|y\|_{\beta} \,\|\widetilde{x}\|_{\beta(a,b)} \big)(b-a)^{\beta} \Big], \\
				G^3_{\beta(a,b)}(f,\widetilde{x}) &=& K \Big[ \|f'\|_\infty + \|f''\|_\infty \|\widetilde{x}\|_{\beta(a,b)} (b-a)^{\beta} \Big].
			\end{eqnarray*}

The first result is Proposition 6.4 of Hu and Nualart \cite{HN}:
	\begin{prop} \label{prop:normbetaf}
		Suppose that $(x,y,x\otimes y)$ and $(\widetilde{x},y,\widetilde{x}\otimes y)$ belong to $M^\beta_{d,m}(0,T)$. Assume that $f$ satisfies {\upshape\bfseries (H1)}. Then, for $0\leq a<b\leq T$,
			\begin{eqnarray*}
			&&	\bigg| \int_a^b \left[f(x_u)-f(\widetilde{x}_u)\right] \,dy_u \bigg| \leq G^1_{\beta(a,b)}(f,x,\widetilde{x},y) \,(b-a)^{2\beta} \|x-\widetilde{x}\|_{\infty(a,b)} \\
				&& + G^2_{\beta(a,b)}(f,x,\widetilde{x},y) \,(b-a)^{2\beta} \|x-\widetilde{x}\|_{\beta(a,b)} 
				+ G^3_{\beta(a,b)}(f,\widetilde{x}) \,(b-a)^{2\beta} \|(x-\widetilde{x})\otimes y\|_{2\beta(a,b)}.
			\end{eqnarray*}
	\end{prop}

We can also deduce the following estimate:  
	\begin{prop} \label{prop:normbetarf}
		Suppose that $(x,y,x\otimes y)$ and $(x_{\cdot-r},y,x_{\cdot-r}\otimes y)$ belong to $M^\beta_{d,m}(0,T)$. Assume that $f$ satisfies {\upshape\bfseries (H1)}. Then, for $0\leq a<b\leq T$,
			\begin{eqnarray} \label{eq:normbetarf}
				\bigg| \int_a^b \left[f(x_u)-f(x_{u-r})\right] \,dy_u \bigg| &\leq& \,G^1_{\beta(a,b)}(f,x,x_{\cdot-r},y) \,(b-a)^{2\beta} \|x-x_{\cdot-r}\|_{\infty(a,b)} \nonumber \\
				&& + \,G^2_{\beta(a,b)}(f,x,x_{\cdot-r},y) \,(b-a)^{2\beta} \|x-x_{\cdot-r}\|_{\beta(a,b)} \nonumber \\
				&& + \,G^3_{\beta(a,b)}(f,x_{\cdot-r}) \,(b-a)^{2\beta} \|(x-x_{\cdot-r})\otimes y\|_{2\beta(a,b)}. \nonumber 
			\end{eqnarray}
	\end{prop}
\begin{proof}
	The proposition is a particular case of Proposition \ref{prop:normbetaf} with $\widetilde{x}\equiv x_{\cdot-r}$. 
\end{proof}

Let us introduce more useful notation:
\begin{eqnarray*}
				 G^4_{\beta(a,b)}(f,x,\widetilde{x},y,z) &=& K \Big[ \|f'\|_\infty \Phi_{\beta(a,b)}(y,z) \\
				 && \hspace{10mm} + \big( \|f''\|_\infty + \|f''\|_\lambda (\|x\|^\lambda_{\beta(a,b)} + \|\widetilde{x}\|^\lambda_{\beta(a,b)}) (b-a)^{\lambda\beta} \big) \\
				 && \hspace{15mm} \times \big( \Phi_{\beta(a,b)}(x,y,z) + \|\widetilde{x}\|_{\beta(a,b)} \Phi_{\beta(a,b)}(y,z) \big) \Big] ,\\
				 G^5_{\beta(a,b)}(f,x,\widetilde{x},y,z) &=& K \Big[ \big( \|f'\|_\infty + \|f''\|_\infty \|\widetilde{x}\|_{\beta(a,b)}(b-a)^{\beta} \big) \Phi_{\beta(a,b)}(y,z)\\
				 && \hspace{10mm} + \|f''\|_\infty \Phi_{\beta(a,b)}(x,y,z) (b-a)^\beta \Big], \\
				 G^6_{\beta(a,b)}(f,\widetilde{x},z) &=& K G^3_{\beta(a,b)}(f,\widetilde{x}) \,\|z\|_{\beta(a,b)}. \\
			\end{eqnarray*}
From the previous results it is possible to prove the following two propositions:  
	\begin{prop} \label{prop:normbetafx}
		Suppose that $(x,y,x\otimes y)$, $(\widetilde{x},y,\widetilde{x}\otimes y)$ and $(y,z,y\otimes z)$ belong to $M^\beta_{d,m}(0,T)$. Assume that $f$ satisfies {\upshape\bfseries (H1)}. Then, for every $0\leq a<b\leq T$,
			\begin{eqnarray} \label{eq:normbetafx}
				 && \bigg| \int_a^b \left[f(x_u)-f(\widetilde{x}_u)\right] \,d_u(y\otimes z)_{\cdot,b} \bigg|  \leq G^4_{\beta(a,b)}(f,x,\widetilde{x},y,z) \,(b-a)^{3\beta} \|x-\widetilde{x}\|_{\infty(a,b)} \nonumber \\
				 && \hspace{30mm} + G^5_{\beta(a,b)}(f,x,\widetilde{x},y,z) \,(b-a)^{3\beta} \|x-\widetilde{x}\|_{\beta(a,b)} \nonumber \\
				 && \hspace{30mm} + G^6_{\beta(a,b)}(f,\widetilde{x},z) \,(b-a)^{3\beta} \,\|(x-\widetilde{x})\otimes y\|_{2\beta(a,b)} .
			\end{eqnarray}
	\end{prop}    
\begin{proof}
	To simplify the proof we will assume $d=m=1$. Observe that from inequalities (\ref{eq_Phi2}) and (\ref{eq_Phi3}) we obtain
		\begin{equation} \label{eq_Phi4}
			\Phi_{\beta(a,b)}(x,y\otimes z) \leq K \Phi_{\beta(a,b)}(x,y,z) (b-a)^\beta
		\end{equation}
	and
		\begin{eqnarray} \label{eq_otimesxyz}
			\big\| (x-\widetilde{x})\otimes(y\otimes z)_{\cdot,b} \big\|_{2\beta(a,b)} &\leq& K \,\Phi_{\beta(a,b)}(y,z) (b-a)^{\beta} \,\| x-\widetilde{x} \|_{\beta(a,b)} \nonumber \\
			&& \hspace{5mm} + K \,\|z\|_{\beta(a,b)} (b-a)^{\beta} \,\| (x-\widetilde{x})\otimes y \|_{2\beta(a,b)}. \nonumber 
		\end{eqnarray}
	The proof of the proposition is obtained applying Proposition \ref{prop:normbetaf} and using inequalities \eqref{eq_Phi2}, \eqref{eq_Phi3}, \eqref{eq_Phi4}. 
\end{proof}

\smallskip    
	\begin{prop} \label{prop:normbetarfx}
		Suppose that $(x,y,x\otimes y)$, $(x_{\cdot-r},y,x_{\cdot-r}\otimes y)$ and $(y,z,y\otimes z)$ belong to $M^\beta_{d,m}(0,T)$. Assume that $f$ satisfies {\upshape\bfseries (H1)}. Then, for $0\leq a<b\leq T$,
			\begin{eqnarray} \label{eq:normbetarf2}
				&& \bigg| \int_a^b \left[f(x_u)-f(x_{u-r})\right] \,d_u(y\otimes z)_{\cdot,b} \bigg| \\
				&& \hspace{25mm} \leq \,G^4_{\beta(a,b)}(f,x,x_{\cdot-r},y,z) \,(b-a)^{3\beta} \|x-x_{\cdot-r}\|_{\infty(a,b)} \nonumber \\
				&& \hspace{30mm} + \,G^5_{\beta(a,b)}(f,x,x_{\cdot-r},y,z) \,(b-a)^{3\beta} \|x-x_{\cdot-r}\|_{\beta(a,b)} \nonumber \\
				&& \hspace{30mm} + \,G^6_{\beta(a,b)}(f,x_{\cdot-r},z) \,(b-a)^{3\beta} \|(x-x_{\cdot-r})\otimes y\|_{2\beta(a,b)}. \nonumber 
			\end{eqnarray}
	\end{prop}
\begin{proof}
	The proposition is a particular case of Proposition \ref{prop:normbetafx} with $\widetilde{x}\equiv x_{\cdot-r}$. 
\end{proof}
  
We conclude this section with a general result on $\beta$-H\"older functions:
	\begin{lemma} \label{lemma_yyr}
	 	Let $y:[0,T]\rightarrow \R^m$ be a $\beta $-H\"older continuous function and $\beta'=\beta-\varepsilon$ for $\varepsilon>0$, then
		 	\begin{eqnarray}
				 \|y-y_{\cdot-r}\|_{\infty(r,T)} &\leq& \|y\|_{\beta} \,r^{\beta}, \label{yyr_inf} \\
				 \|y-y_{\cdot-r}\|_{\beta'(r,T)} &\leq& 2 \|y\|_{\beta} \,r^{\varepsilon}. \label{yyr_beta'} 
		 	\end{eqnarray}
	\end{lemma}
\begin{proof}
 	On one hand,
	 	\begin{eqnarray*}
	 		\|y-y_{\cdot-r}\|_{\infty(r,T)} = \sup_{t\in[r,T]} \frac{|y_t-y_{t-r}|}{r^\beta} \cdot r^\beta \leq \|y\|_{\beta} r^\beta.
	 	\end{eqnarray*}
	On the other hand, we have
		\begin{eqnarray}
			&& \sup_{\substack{s<t\in[r,T] \\ t-s\leq r}} \frac{\big| (y-y_{\cdot-r})_t-(y-y_{\cdot-r})_s \big|}{(t-s)^{\beta'}}\label{aa1} \\
		 	&& \hspace{5mm} \leq \sup_{\substack{s<t\in[r,T] \\ t-s\leq r}} \frac{|y_t-y_s|}{(t-s)^{\beta}} \cdot \frac{(t-s)^{\beta}}{(t-s)^{\beta'}} + \sup_{\substack{s<t\in[r,T] \\ t-s\leq r}} \frac{|y_{t-r}-y_{s-r}|}{(t-s)^{\beta}} \cdot \frac{(t-s)^{\beta}}{(t-s)^{\beta'}}  \leq 2 \|y\|_{\beta} \,r^{\varepsilon}\nonumber
		\end{eqnarray}
	and
		\begin{eqnarray}
			&& \sup_{\substack{s<t\in[r,T] \\ t-s\geq r}} \frac{\big| (y-y_{\cdot-r})_t-(y-y_{\cdot-r})_s \big|}{(t-s)^{\beta'}}\label{aa2} \\
		 	&& \hspace{5mm} \leq \sup_{\substack{s<t\in[r,T] \\ t-s\geq r}} \frac{|y_t-y_{t-r}|}{r^{\beta}} \cdot \frac{r^{\beta}}{(t-s)^{\beta'}} + \sup_{\substack{s<t\in[r,T] \\ t-s\geq r}} \frac{|y_s-y_{s-r}|}{r^{\beta}} \cdot \frac{r^{\beta}}{(t-s)^{\beta'}}  \leq 2 \|y\|_{\beta} \,r^{\varepsilon}.\nonumber
		\end{eqnarray}
The proof finishes putting together (\ref{aa1}) and (\ref{aa2}).
\end{proof}

\section{Estimates of the solutions} \label{dSDE_EstimatesSol}

In this Section we get some estimates on the solutions of our equations. Let us recall that $\varepsilon>0$  with $\beta-2\varepsilon>0$ and $\lambda>\frac{1}{\beta-\varepsilon}-2$. Recall also that $\beta'=\beta-\varepsilon$.

First of all, let us introduce
$\widehat{x}^r_t=x^r_{t-r}$ where  $x^r$ is the solution of \eqref{eq:delay}. Then  $(\widehat{x}^r\otimes y)_{s,t}$  can be expressed  as follows:
	\begin{itemize}
		\item for $s<t\in[0,r):$ 
\begin{eqnarray}
			(\widehat{x}^r\otimes y)_{s,t} = (\eta_{\cdot-r}\otimes y)_{s,t} = \int_s^t (y_t-y_u) \,d\eta_{u-r},
\label{eq:delay-rxy1}\end{eqnarray}
		\item for $s\in[0,r)$ and $t\in[r,T]$,
			\begin{eqnarray} \label{eq:delay-rxy2}
				(\widehat{x}^r\otimes y)_{s,t} &=& (\eta_{\cdot-r}\otimes y)_{s,r} + \int_r^t (y_t-y_u) b(u-r,\widehat{x}^r_u) \,du \nonumber \\
				&& \hspace{5mm} + \int_r^t \sigma(\widehat{x}^r_{u-r}) \,d_u(y_{\cdot-r}\otimes y)_{\cdot,t} + (\eta_0-\eta_{s-r})\otimes(y_t-y_r), \nonumber
			\end{eqnarray}
		\item for $s<t\in[r,T]:$\quad 
			$(\widehat{x}^r\otimes y)_{s,t} =\int_s^t (y_t-y_u) b(u-r,\widehat{x}^r_u) \,du + \int_s^t \sigma(\widehat{x}^r_{u-r}) \,d_u(y_{\cdot-r}\otimes y)_{\cdot,t}. $
	\end{itemize}

We will prove that the norms $\|\widehat{x}^r\|_{\beta'}$ and $\|\widehat{x}^r\otimes y\|_{2\beta'}$ are bounded and their upper bound does not depend on $r$. To this aim, the following lemma will be useful:
	\begin{lemma} \label{lemma_cotesxr}
		Let $(\eta_{\cdot-r},y,\eta_{\cdot-r}\otimes y)\in M^\beta_{d,m}(0,r)$ and $(y_{\cdot-r},y,y_{\cdot-r}\otimes y)\in M^\beta_{d,d}(r,T)$. Let $(x^r,y,x^r\otimes y)\in M^\beta_{d,m}(0,T)$ be the solution of the equation \eqref{eq:delay}. Then,
			\begin{equation} \label{cotaxr1}
			  \|\widehat{x}^r\|_{\beta'} \leq \|\widehat{x}^r\|_{\beta'(0,r)} + \|\widehat{x}^r\|_{\beta'(r,T)}
			\end{equation}
		and
			\begin{equation} \label{cotaxr2}
			  \|\widehat{x}^r\otimes y\|_{2\beta'} \leq \|\widehat{x}^r\otimes y\|_{2\beta'(0,r)} + \|\widehat{x}^r\otimes y\|_{2\beta'(r,T)} + \|\eta\|_{\beta'(-r,0)}\|y\|_{\beta'}.
			\end{equation}
	\end{lemma}
  
\begin{proof}
	On one hand, observe that
		\begin{equation*}
			\|\widehat{x}^r\|_{\beta'} \leq \max \Bigg( \sup_{0\leq s<t<r} \frac{|\widehat{x}^r_t-\widehat{x}^r_s|}{(t-s)^{\beta'}} \,, \sup_{0\leq s<r\leq t\leq T} \frac{|\widehat{x}^r_t-\widehat{x}^r_s|}{(t-s)^{\beta'}} \,, \sup_{r\leq s<t\leq T} \frac{|\widehat{x}^r_t-\widehat{x}^r_s|}{(t-s)^{\beta'}} \Bigg), \\
		\end{equation*}
	and
		\begin{eqnarray*}
			\sup_{0\leq s<r\leq t\leq T} \frac{|\widehat{x}^r_t-\widehat{x}^r_s|}{(t-s)^{\beta'}} 
%&\leq& \sup_{0\leq s<r\leq t\leq T} \frac{|\widehat{x}^r_t-\widehat{x}^r_r|}{(t-s)^{\beta'}} + \sup_{0\leq s<r\leq t\leq T} \frac{|\widehat{x}^r_r-\widehat{x}^r_s|}{(t-s)^{\beta'}} \\
%			&\leq& \sup_{r\leq t\leq T} \frac{|\widehat{x}^r_t-\widehat{x}^r_r|}{(t-r)^{\beta'}} + \sup_{0\leq s<r} \frac{|\widehat{x}^r_r-\widehat{x}^r_s|}{(r-s)^{\beta'}} \\
			&\leq& \sup_{r\leq s<t\leq T} \frac{|\widehat{x}^r_t-\widehat{x}^r_s|}{(t-s)^{\beta'}} + \sup_{0\leq s<t<r} \frac{|\widehat{x}^r_t-\widehat{x}^r_s|}{(t-s)^{\beta'}}.
		\end{eqnarray*}
	So we easily get \eqref{cotaxr1}.

	On the other hand, observe that from the multiplicative property we obtain
		\begin{eqnarray*}
			\sup_{0\leq s<r\leq t\leq T} \frac{|(\widehat{x}^r\otimes y)_{s,t}|}{(t-s)^{2\beta'}} &\leq& \sup_{0\leq s<r\leq t\leq T} \bigg[ \frac{|(\widehat{x}^r\otimes y)_{s,r}|}{(t-s)^{2\beta'}} + \frac{|(\widehat{x}^r\otimes y)_{r,t}|}{(t-s)^{2\beta'}} \\
			&& \hspace{30mm} + \frac{|(\widehat{x}^r_r-\widehat{x}^r_s)\otimes(y_t-y_r)|}{(t-s)^{2\beta'}} \bigg],
		\end{eqnarray*}
	and using the same argument as before \eqref{cotaxr2} follows easily.
\end{proof}

\smallskip

Now we can give the following result:
	\begin{prop} \label{cotesxr}
		Let $(\eta_{\cdot-r},y,\eta_{\cdot-r}\otimes y)\in M^\beta_{d,m}(0,r)$ and $(y_{\cdot-r},y,y_{\cdot-r}\otimes y)\in M^\beta_{d,d}(r,T)$ for all $r \le r_0$. Assume that $\sigma$ and $b$ satisfy {\upshape\bfseries (H1)} and {\upshape\bfseries (H2)} respectively, and both satisfy {\upshape\bfseries (H3)}. Let $(x^r,y,x^r\otimes y)\in M^\beta_{d,m}(0,T)$ be the solution of the equation \eqref{eq:delay}. Assume also that  $\|\eta\|_{\beta(-r_0,0)}<\infty$ and $\sup_{r\leq r_0}\|\eta_{\cdot -r}\otimes y\|_{2\beta(0,r)}<\infty$. Then, for $r\leq r_0$, we have the following estimates:
			\begin{eqnarray}
				\|\widehat{x}^r\|_{\infty(0,T+r)} &\leq& M_{\eta,y}, \label{normsupxr} \\
				\|\widehat{x}^r\|_{\beta'(0,T+r)} &\leq& K \rho_{\eta,b,\sigma} \Lambda_y (1+2M_{\eta,y}), \label{normbetaxr} \\
				\|\widehat{x}^r\otimes y\|_{2\beta'(0,T+r)} &\leq& K \rho_{\eta,b,\sigma} \Lambda_y \big( 2 + (T+r_0) (K \rho_{\eta,b,\sigma} \Lambda_y)^{\frac1\beta} \big), \label{norm2betaxroy}
			\end{eqnarray}
		where $K\geq1$ and
			\begin{eqnarray} \label{rho}
				\rho_{\eta,b,\sigma} :&=&2\|\eta\|_{\beta(-r_0,0)} + \|b\|_\infty T^{1-\beta} + \|\sigma\|_\infty + \|\sigma'\|_\infty + \|\sigma'\|_\lambda,\\
			 \label{Lambda}
				\Lambda_y :&=& \|y\|_\beta + \max(1,\|y\|_\beta^2+\|y\otimes y\|_{2\beta}),
			\end{eqnarray}
		and
			\begin{equation*} %\label{M}
				M_{\eta,y} := |\eta_0| + (T+r_0) (K\rho_{\eta,b,\sigma}\Lambda_y)^{\frac1\beta} + 1.
			\end{equation*}
	\end{prop}
  
\begin{proof}
    To simplify the proof we will assume $d=m=1$. Assume also that $r\leq r_0$. 
    
    \vskip 5pt
    First we observe that, if $\|\eta\|_{\beta(-r_0,0)}<C$ and $\sup_{r\leq r_0}\|\eta_{\cdot -r}\otimes y\|_{2\beta(0,r)}<C'$, with $C$ and $C'$ two positive constants, then $\|\eta\|_{\beta'(-r_0,0)}<Cr_0^\varepsilon$ and $\sup_{r\leq r_0}\|\eta_{\cdot -r}\otimes y\|_{2\beta'(0,r)}<C'r_0^{2\varepsilon}$.
    
    \vskip 5pt
    Secondly, notice that by \eqref{eq:delay-rxy1} 
	    \[ \|\eta_{\cdot-r}\otimes y\|_{2\beta(0,r)} = \sup_{s,t\in[0,r)} \frac{\big| \int_s^t (y_t-y_u) \,d\eta_{u-r} \big|}{(t-s)^{2\beta}} \leq \|\eta_{\cdot-r}\|_{\beta(0,r)} \|y\|_{\beta} \leq \|\eta\|_{\beta(-r_0,0)} \|y\|_{\beta}. \]
%    If $\eta$ is differentiable and monotone and $\|\eta\|_{\beta(-r_0,0)}<\infty$, then 
%	    \[ \|\eta_{\cdot-r}\otimes y\|_{2\beta(0,r)} < \infty. \]
 %   Moreover, for all $\beta'<\beta$, it is true that $\|\eta\|_{\beta'(-r_0,0)}<\infty$ and
 %   	\[ \|\eta_{\cdot-r}\otimes y\|_{2\beta'(0,r)} \leq \|\eta\|_{\beta'(-r_0,0)} \|y\|_{\beta'} < \infty. \]
    To prove the result we will follow the ideas of Theorem 4.1 of \cite{BN}. Consider the mapping $J:M_{1,1}^\beta(0,T+r)\rightarrow M_{1,1}^\beta(0,T+r)$ given by $J(\widehat{x}^r,y,\widehat{x}^r\otimes y)=(J_1,y,J_2)$ where $J_1$ and $J_2$ are the right-hand sides of the definition of $\widehat{x}^r$  and $(\widehat{x}^r\otimes y)$  respectively:
		\begin{eqnarray} \label{eqJ1}
			&& J_1(\widehat{x}^r,y,\widehat{x}^r\otimes y)(t) =
			\left\{ \begin{array}{ll}
				\displaystyle \eta_{t-r}, & 0\leq t<r \nonumber\\
				\displaystyle \eta_0 + \int_0^{t-r} b(u,\widehat{x}^r_{u+r}) \,du + \int_0^{t-r} \sigma(\widehat{x}^r_u) \,dy_u, & r\leq t\leq T
			\end{array} \right. \nonumber \\
		\end{eqnarray}
		\begin{eqnarray} \label{eqJ2}
			&& J_2(\widehat{x}^r,y,\widehat{x}^r\otimes y)(s,t) \nonumber \\
			&& \hspace{5mm} = \left\{
			\begin{array}{ll}
				\displaystyle \int_s^t (y_t-y_u) \,d\eta_{u-r}, & 0\leq s\leq t<r \\
				\displaystyle (\eta_0-\eta_{s-r})\otimes(y_t-y_r) + \int_s^r (y_r-y_u) \,d\eta_{u-r}  +& \int_r^t (y_t-y_u) b(u-r,\widehat{x}^r_u) \,du \\
				\displaystyle \hspace{15mm} + \int_r^t \sigma(\widehat{x}^r_{u-r}) \,d_u(y_{\cdot-r}\otimes y)_{\cdot,t}, & 0\leq s<r\leq t\leq T \\
				\displaystyle \int_s^t (y_t-y_u) b(u-r,\widehat{x}^r_u) \,du \\
				\displaystyle \hspace{15mm} + \int_s^t \sigma(\widehat{x}^r_{u-r}) \,d_u(y_{\cdot-r}\otimes y)_{\cdot,t} & r\leq s\leq t\leq T
			\end{array} \right. \nonumber 
		\end{eqnarray}
	Remark that this mapping is well-defined because $(J_1,\,y,\,J_2)$ is a real-valued $\beta-$H\"older continuous multiplicative functional for each $(\widehat{x}^r,y,\widehat{x}^r\otimes y)\in M_{1,1}^\beta(0,T)$.

	\vskip 5pt
    Now we bound the H\"older norms of $J_1$ and $J_2$ using Proposition \ref{prop1HN} and Proposition \ref{prop2HN}. Let $s<t\in[0,T]$, we have
	    \begin{itemize}
	    	\item for $s<t\in[0,r)$
		      	\begin{eqnarray}
			        \|J_1\|_{\beta(s,t)} &\leq& \|\eta\|_{\beta(-r_0,0)}, \label{cotaJ1_0r} \\
			        \|J_2\|_{2\beta(s,t)} &\leq& \|\eta\|_{\beta(-r_0,0)} \|y\|_{\beta}, \label{cotaJ2_0r}
		      	\end{eqnarray}
	      	\item for $s<t\in[r,T]$
		      	\begin{eqnarray}
			      	& &\|J_1\|_{\beta(s,t)} \leq \|b\|_{\infty} (t-s)^{1-\beta} + K \|\sigma\|_{\infty} \,\|y\|_{\beta} \nonumber \\
			      	&& \hspace{2mm} + K \left( \|\sigma'\|_\infty + \|\sigma'\|_\lambda \|\widehat{x}^r_{\cdot-r}\|^\lambda_{\beta(s,t)} (t-s)^{\lambda\beta} \right) \times \Phi_{\beta(s,t)}(\widehat{x}^r_{\cdot-r},y_{\cdot-r}) (t-s)^{\beta},  \label{cotaJ1_rT} \\
			      	& &\|J_2\|_{2\beta(s,t)} \leq \|b\|_{\infty} \,\|y\|_{\beta} (t-s)^{1-\beta} + K \|\sigma\|_{\infty} \Phi_{\beta(s,t)}(y_{\cdot-r},y) \nonumber \\
			      	&& \hspace{2mm} + K \left( \|\sigma'\|_\infty + \|\sigma'\|_\lambda \|\widehat{x}^r_{\cdot-r}\|^\lambda_{\beta(s,t)} (t-s)^{\lambda\beta} \right)   \Phi_{\beta(s,t)}(\widehat{x}^r_{\cdot-r},y_{\cdot-r},y) (t-s)^{\beta}, \label{cotaJ2_rT}
		      	\end{eqnarray}
	      	\item for $s\in[0,r)$ and $t\in[r,T]$
		      	\begin{eqnarray}
			        \|J_1\|_{\beta(s,t)} &\leq& \|J_1\|_{\beta(s,r)}+ \|J_1\|_{\beta(r,t)} 
			        \leq \|\eta\|_{\beta(-r_0,0)} + \|b\|_{\infty} (t-r)^{1-\beta} + K \|\sigma\|_{\infty} \,\|y\|_{\beta} \nonumber \\
			        && \hspace{1mm} + K \left( \|\sigma'\|_\infty + \|\sigma'\|_\lambda \|\widehat{x}^r_{\cdot-r}\|^\lambda_{\beta(r,t)} (t-r)^{\lambda\beta} \right)  \times \Phi_{\beta(r,t)}(\widehat{x}^r_{\cdot-r},y_{\cdot-r}) (t-r)^{\beta}, \nonumber% \\ \label{cotaJ1_0rT} 
\end{eqnarray}
\begin{eqnarray}
			        \|J_2\|_{2\beta(s,t)} &\leq& \|J_2\|_{2\beta(s,r)}+ \|J_2\|_{2\beta(r,t)} \nonumber \\
			        &\leq& 2 \|\eta\|_{\beta(-r_0,0)} \|y\|_{\beta} + \|b\|_{\infty} \,\|y\|_{\beta} (t-r)^{1-\beta} + K \|\sigma\|_{\infty} \Phi_{\beta(r,t)}(y_{\cdot-r},y) \nonumber \\
			        && \hspace{2mm} + K \left( \|\sigma'\|_\infty + \|\sigma'\|_\lambda \|\widehat{x}^r_{\cdot-r}\|^\lambda_{\beta(r,t)} (t-r)^{\lambda\beta} \right)  \Phi_{\beta(r,t)}(\widehat{x}^r_{\cdot-r},y_{\cdot-r},y) (t-r)^{\beta}. \nonumber %\\ \label{cotaJ2_0rT}
		      	\end{eqnarray}
	    \end{itemize}
	For $s<t\in[r,T]$, we set
		\begin{equation*}
			(\widehat{x}^r_{\cdot-r}\otimes y_{\cdot-r})_{s,t} := (\widehat{x}\otimes y)_{s-r,t-r}.
		\end{equation*}
	In Section 5 of \cite{BMR} it is proved that it is a $\beta$-H\"older continuous multiplicative functional. \\
	We proceed dividing the proof in two steps.

	\vskip 5pt
	{\underline{Step 1:}} We will find a set $C^y$ of elements $(\widehat{x}^r,y,\widehat{x}^r\otimes y)\in M_{1,1}^\beta(0,T)$ such that $J(C^y)\subset C^{y}$. Recall definitions of $\rho_{\eta,b,\sigma}$ and $\Lambda_y$ from \eqref{rho} and \eqref{Lambda}, respectively, and set
		\begin{equation*}
			\widetilde{\Delta}_y := \big( K\rho_{\eta,b,\sigma} \Lambda_y \big)^{-\frac1\beta}.
		\end{equation*}
	Let $C^y$ be the set of elements $(\widehat{x}^r,y,\widehat{x}^r\otimes y)\in M_{1,1}^\beta(0,T)$ satisfying the following conditions:
		\begin{eqnarray}
			\|\widehat{x}^r\|_{\infty} &\leq& M_{\eta,y}, \nonumber\\%\label{cond1} \\
			\sup_{0<t-s\leq\widetilde{\Delta}_y} \|\widehat{x}^r\|_{\beta(s,t)} &\leq& K\rho_{\eta,b,\sigma} (\|y\|_\beta+1), \label{cond2} \\
			\sup_{0<t-s\leq\widetilde{\Delta}_y} \|\widehat{x}^r\otimes y\|_{2\beta(s,t)} &\leq& K\rho_{\eta,b,\sigma} (\|y\|_\beta+\|y\|_\beta^2+\|y_{\cdot-r}\otimes y\|_{2\beta}). \label{cond3}
		\end{eqnarray}
	We take $s,t\in[0,T]$ such that
		\begin{equation} \label{cota1}
			0 < t-s \leq \widetilde{\Delta}_y,
		\end{equation}
	and then we have
		\begin{equation} \label{cota2}
			(t-s)^\beta \leq \widetilde{\Delta}_y^{\beta} \leq \frac{1}{K\rho_{\eta,b,\sigma}(\|y\|_\beta+1)}
		\end{equation}
	and
		\begin{equation} \label{cota3}
			(t-s)^\beta \leq \widetilde{\Delta}_y^{\beta} \leq \frac{1}{K\rho_{\eta,b,\sigma}(\|y\|_\beta+\|y\|_\beta^2+\|y_{\cdot-r}\otimes y\|_{2\beta})}.
		\end{equation}
	Suppose that $(\widehat{x}^r,y,\widehat{x}^r\otimes y)\in C^y$, then using \eqref{cond2}, \eqref{cota2} and \eqref{cond3}, \eqref{cota3} respectively, we have
		\begin{eqnarray}
			(t-s)^\beta \|\widehat{x}^r\|_{\beta(s,t)} &\leq& 1, \label{cota7} \\
			(t-s)^\beta \|\widehat{x}^r\otimes y\|_{2\beta(s,t)} &\leq& 1.  \label{cota8}
		\end{eqnarray}
	Now observe that, if $s,t\in[r,T]$ satisfy \eqref{cota1}, then $s-r,t-r\in[0,T]$ also satisfy this condition. As a consequence,
		\begin{equation} \label{cota7-r}
			(t-s)^\beta \|\widehat{x}^r_{\cdot-r}\|_{\beta(s,t)} \leq 1
		\end{equation}
	and
		\begin{equation} \label{cota8-r}
			(t-s)^\beta \|\widehat{x}^r_{\cdot-r}\otimes y_{\cdot-r}\|_{2\beta(s,t)} \leq 1.
		\end{equation}
	From the last inequality it easily follows that
		\begin{eqnarray} \label{cotaPhi}
			&& \Phi_{\beta(s,t)}(\widehat{x}^r_{\cdot-r},y_{\cdot-r},y) (t-s)^{\beta} \nonumber \\
			&& \hspace{15mm} = \Big[ \|\widehat{x}^r_{\cdot-r}\|_{\beta(s,t)} \|y_{\cdot-r}\|_{\beta(s,t)} \|y\|_{\beta(s,t)} + \|y\|_{\beta(s,t)} \|\widehat{x}^r_{\cdot-r}\otimes y_{\cdot-r}\|_{2\beta(s,t)} \nonumber \\
			&& \hspace{25mm} + \|\widehat{x}^r_{\cdot-r}\|_{\beta(s,t)} \|y_{\cdot-r}\otimes y\|_{2\beta(s,t)} \Big] (t-s)^{\beta} \nonumber \\
			&& \hspace{15mm} \leq \|y\|_\beta+\|y\|_\beta^2+\|y_{\cdot-r}\otimes y\|_{2\beta}.
		\end{eqnarray}
Observe also  that if $s\in[0,r)$ and $t\in[r,T]$ satisfy \eqref{cota1}, then $t-r\leq\widetilde{\Delta}_y$ and all the previous inequality are satisfied if we change the interval $(s,t) $ to the interval $(r,t)$ for $t\in[r,T]$. \\
	By expressions from \eqref{cotaJ1_0r} to \eqref{cotaJ2_rT} and from \eqref{cota7} to \eqref{cotaPhi} we easily get that
		\begin{eqnarray} \label{cotaJ1_2}
			\|J_1\|_{\beta(s,t)} &\leq& \|\eta\|_{\beta(-r_0,0)} + \|b\|_{\infty} T^{1-\beta} + K \|\sigma\|_{\infty} \,\|y\|_{\beta} \nonumber \\
			&& \hspace{5mm} + K \left( \|\sigma'\|_\infty + \|\sigma'\|_\lambda \right) (\|y\|_{\beta} + 1) \nonumber \\
			&\leq& K \rho_{\eta,b,\sigma} (\|y\|_{\beta} + 1)
		\end{eqnarray}
	and
		\begin{eqnarray} \label{cotaJ2_2}
			\|J_2\|_{2\beta(s,t)} &\leq& 2 \|\eta\|_{\beta(-r_0,0)} \|y\|_{\beta} + \|b\|_{\infty} \,\|y\|_{\beta} T^{1-\beta} + K \|\sigma\|_{\infty} (\|y\|_\beta^2 + \|y_{\cdot-r}\otimes y\|_{2\beta}) \nonumber \\
			&& \hspace{5mm} + K \left( \|\sigma'\|_\infty + \|\sigma'\|_\lambda \right) (\|y\|_\beta + \|y\|_\beta^2 + \|y_{\cdot-r}\otimes y\|_{2\beta}) \nonumber \\
			&\leq& K \rho_{\eta,b,\sigma} (\|y\|_\beta + \|y\|_\beta^2 + \|y_{\cdot-r}\otimes y\|_{2\beta}) \nonumber
		\end{eqnarray}
	where $K\geq1$. \\
	It only remains to prove that $\|J_1\|_\infty \leq M_{\eta,y}$. Set $N=\big[(T+r)\widetilde{\Delta}_y^{-1}\big]+1$ and define the partition $t_0=0<t_1<\dots<t_N=T+r$, where $t_i=i\widetilde{\Delta}_y$ for $i=0,\dots,N-1$. The estimates \eqref{cota2} and \eqref{cotaJ1_2} imply
		\[ \sup_{u\in[t_{i-1},t_i]} |(J_1)_u| \leq |(J_1)_{t_{i-1}}| + (t_i-t_{i-1})^\beta \|J_1\|_{\beta(t_{i-1},t_i)} \leq |(J_1)_{t_{i-1}}|+1. \]
	Moreover,
		\[ \sup_{u\in [0,t_i]} |(J_1)_u| \leq \sup_{u\in[0,t_{i-1}]} |(J_1)_u|+1, \]
	and iterating we finally get that
		\[ \sup_{u\in [0,T]} |(J_1)_u| \leq |\eta_0|+N \leq |\eta_0| + T\widetilde{\Delta}_y^{-1} + 1 = M_{\eta,y}. \]
	Hence, $(J_1,y,J_2)\in C^y$.

	\smallskip
	{\underline{Step 2:}} We find a bound for the H\"older norms of $\widehat{x}^r$ and $(\widehat{x}^r\otimes y)$. \\
	We can construct a sequence of functions $\widehat{x}^{r(n)}$ and $(\widehat{x}^r\otimes y)^{(n)}$ such that, 
		$\widehat{x}^{r(0)} = \eta_0 $  and $ (\widehat{x}^r\otimes y)^{(0)}=0$
	and
		\begin{eqnarray*}
			\widehat{x}^{r(n)} &=& J_1\left(\widehat{x}^{r(n-1)}, y, \left(\widehat{x}^r\otimes y\right)^{(n-1)}\right),\\
			\left(\widehat{x}^r\otimes y\right)^{(n)} &=& J_2\left(\widehat{x}^{r(n-1)}, y, \left(\widehat{x}^r\otimes y\right)^{(n-1)}\right).
		\end{eqnarray*}
	Notice that $\big(\widehat{x}^{r(0)}, y,(\widehat{x}^r\otimes y)^{(0)}\big)\in C^y$ and, since we have proved in Step 1 that $J\big(C^y\big)\subset C^y$, we have that $\big(\widehat{x}^{r(n)}, y,(\widehat{x}^r\otimes y)^{(n)}\big)\in C^y$ for each $n$. We estimate $\|\widehat{x}^{r(n)}\|_{\beta}$ as follows:
		\begin{eqnarray}
			\|\widehat{x}^{r(n)}\|_{\beta} &\leq& \sup_{\substack{0\leq s<t\leq T \\ t-s\leq\widetilde{\Delta}_y}} \frac{|\widehat{x}^{r(n)}_t-\widehat{x}^{r(n)}_s|}{(t-s)^\beta} + \sup_{\substack{0\leq s<t\leq T \\
					t-s\geq\widetilde{\Delta}_y}} \frac{|\widehat{x}^{r(n)}_t-\widehat{x}^{r(n)}_s|}{(t-s)^\beta} \nonumber \\
			&\leq& K \rho_{\eta,b,\sigma} (\|y\|_\beta+1) + 2 \widetilde{\Delta}_y^{-\beta} \|\widehat{x}^{r(n)}\|_{\infty} \nonumber \\
		%	&\leq& K \rho_{\eta,b,\sigma} (\|y\|_\beta+1)+ 2 \widetilde{\Delta}_y^{-\beta} M_{\eta,y} \nonumber \\
			&\leq& K \rho_{\eta,b,\sigma} \Lambda_y (1+2M_{\eta,y}). \label{cota9}
		\end{eqnarray} 
	
	This implies that the sequence of functions $\widehat{x}^{r(n)}$ is equicontinuous and bounded in $C^\beta(0,T)$ and the upper bound does not depend on $r$. So, there exists a subsequence which converges in the $\beta'$-H\"older norm if $\beta'<\beta$ and such that the upper bound of the $\beta'$-H\"older norm does not depend on $r$.
	
	\smallskip
	In a similar way we obtain the same result for $(\widehat{x}^r\otimes y)^{(n)}$. From inequality \eqref{cota8} we obtain that
$$
			\sup_{t_{i-1}\leq s<t\leq t_i} |(\widehat{x}^r\otimes y)^{(n)}_{s,t}| \leq \|(\widehat{x}^r\otimes y)^{(n)}\|_{2\beta(t_{i-1},t_i)} (t_i-t_{i-1})^{2\beta} \leq \widetilde{\Delta}_y^\beta
$$
	and
		\begin{equation*}
			\sup_{0\leq s<t\leq T} |(\widehat{x}^r\otimes y)^{(n)}_{s,t}| \leq N \widetilde{\Delta}_y^\beta \leq T\widetilde{\Delta}_y^{\beta-1} + \widetilde{\Delta}_y^\beta.
		\end{equation*}
	As for \eqref{cota9}, we estimate $\|(\widehat{x}^r\otimes y)^{(n)}\|_{2\beta}$ as follows:
		\begin{eqnarray}
			\|(\widehat{x}^r\otimes y)^{(n)}\|_{2\beta} &\leq& K \rho_{\eta,b,\sigma} (\|y\|_\beta^2+\|y\|_\beta+\|y_{\cdot-r}\otimes y\|_{2\beta}) + T\widetilde{\Delta}_y^{-\beta-1} + \widetilde{\Delta}_y^{-\beta} \nonumber \\
		&\leq& K \rho_{\eta,b,\sigma} \Lambda_y \big( 2 + (T+r_0) (K \rho_{\eta,b,\sigma} \Lambda_y)^{\frac1\beta} \big).\nonumber
		\end{eqnarray}
	This implies that the sequence of functions $\left(\widehat{x}^r\otimes y\right)^{(n)}$ is bounded and equicontinuous in the set of functions $2\beta$-H\"older continuous on $\Delta_T$, and the upper bound does not depend on $r$. So, there exists a subsequence which converges in the $\beta'$-H\"older norm if $\beta'<\beta$ and such that the upper bound of the $\beta'$-H\"older norm does not depend on $r$.

	\smallskip
	Now as $n$ tends to infinity it is easy to see that the limit is a solution, and the limit defines a $\beta-$H\"older continuous multiplicative functional $(\widehat{x}^r,y,\widehat{x}^r\otimes y)$ and this functional satisfies estimates \eqref{normsupxr}, \eqref{normbetaxr} and \eqref{norm2betaxroy}. 
\end{proof}

	\begin{rem} \label{obs_cotes}
		\normalfont In Proposition \ref{cotesxr} it is proved that $\|\widehat{x}^r\|_{\beta'(0,T)} \leq K \rho_{\eta,b,\sigma} \Lambda_y (1+2M_{\eta,y})$, so we have the same bound for $\|x^r\|_{\beta'(r)}$. Moreover, using the ideas in the proof of Proposition \ref{cotesxr} it is possible to prove that $\|x^r\otimes y\|_{2\beta'}$ is bounded and its bound does not depend on $r$.
	\end{rem}

\smallskip

We are also interesting about the behavior of $(x^r-\widehat{x}^r)$ when $r$ tends to zero.

We can write  $(x-x^r)_t$ as follows
  	\begin{eqnarray} \label{def:difference2}
    	(x-x^r)_t &=& \int_0^t \big[b(u,x_u)-b(u,x^r_u)\big] \,du + \int_0^t \big[\sigma(x_u)-\sigma(x^r_u)\big] \,dy_u \nonumber \\
    	&& \hspace{10mm} + \int_0^t \big[\sigma(x^r_u)-\sigma(x^r_{u-r})\big] \,dy_u.
  	\end{eqnarray}
Following the ideas in Section 4 of \cite{BMR}, let us write $\big( (x-x^r)\otimes y \big)_{s,t}$ for $s,t\in[0,T]$:
	\begin{eqnarray} \label{def:differenceotimes}
	& &\big( (x-x^r)\otimes y \big)_{s,t} = \int_s^t (y_t-y_u) \big[b(u,x_u)-b(u,x^r_u)\big] \,du \nonumber \\
	&& \hspace{10mm} + \int_s^t \big[\sigma(x_u)-\sigma(x^r_u)\big] \,d_u(y\otimes y)_{\cdot,t}
	 + \int_s^t \big[\sigma(x^r_u)-\sigma(x^r_{u-r})\big] \,d_u(y\otimes y)_{\cdot,t}.
	\end{eqnarray}

%Now we need to distinguish two cases to define $(x^r-\widehat{x}^r)_t$:
%	\begin{itemize}
%	\item for $t\in[0,r)$,
%		\begin{eqnarray} \label{eq:diffdelay0r}%
%			(x^r-\widehat{x}^r)_t = \eta_0 - \eta_{t-r} + \int_0^t b(u,x^r_u) \,du + \int_0^t \sigma(\eta_{u-r}) \,dy_u,
%		\end{eqnarray}
%	\item for $t\in[r,T]$,
%		\begin{eqnarray} \label{eq:diffdelayrT}
%			(x^r-\widehat{x}^r)_t = \int_{t-r}^t b(u,x^r_u) \,du + \int_{t-r}^t \sigma(x^r_{u-r}) \,dy_u.
%		\end{eqnarray}
%	\end{itemize}
It is also useful to writte the following expressions
	\begin{itemize}
		\item for $s<t\in[0,r)$,
			\begin{eqnarray} \label{eq:diffdelayst0r}
				(x^r-\widehat{x}^r)_t - (x^r-\widehat{x}^r)_s = \eta_{s-r} - \eta_{t-r} + \int_s^t b(u,x^r_u) \,du + \int_s^t \sigma(\eta_{u-r}) \,dy_u, \nonumber \\
			\end{eqnarray}
		\item for $s\in[0,r)$ and $t\in[r,T]$,
			\begin{eqnarray} \label{eq:diffdelayst0rT}
				(x^r-\widehat{x}^r)_t - (x^r-\widehat{x}^r)_s &=& \eta_{s-r} - \eta_0 + \int_{t-r}^t b(u,x^r_u) \,du - \int_0^s b(u,x^r_u) \,du \qquad \nonumber\\
				&& \hspace{5mm} + \int_{t-r}^t \sigma(x^r_{u-r}) \,dy_u - \int_0^s \sigma(\eta_{u-r}) \,dy_u, \nonumber
			\end{eqnarray}
		\item for $s<t\in[r,T]$,
			\begin{eqnarray}
			& &	(x^r-\widehat{x}^r)_t - (x^r-\widehat{x}^r)_s =\label{eq:diffdelaystrT2}\\
%& &  = \int_{t-r}^t b(u,x^r_u) \,du - \int_{s-r}^s b(u,x^r_u) \,du
%				 + \int_{t-r}^t \sigma(x^r_{u-r}) \,dy_u - \int_{s-r}^s \sigma(x^r_{u-r}) \,dy_u \label{eq:diffdelaystrT1} \\%
				& & = \int_s^t b(u,x^r_u) \,du - \int_{s-r}^{t-r} b(u,x^r_u) \,du + \int_s^t \sigma(x^r_{u-r}) \,dy_u - \int_{s-r}^{t-r} \sigma(x^r_{u-r}) \,dy_u. \nonumber
			\end{eqnarray}
	\end{itemize}
Finally, following the ideas in Section 4 of \cite{BMR}, we define 
	\[ \big( (x^r-\widehat{x}^r)\otimes y \big)_{s,t}:=(x^r\otimes y)_{s,t} - (\widehat{x}^r\otimes y)_{s,t}, \]
that is:
	\begin{itemize}
		\item for $s<t\in[0,r)$,
			\begin{eqnarray} \label{eq:diffdelayx0r}
				\big( (x^r-\widehat{x}^r)\otimes y \big)_{s,t} &=& \int_s^t (y_u-y_t) \,d\eta_{u-r} + \int_s^t (y_t-y_u) b(u,x^r_u) \,du \nonumber \\
				&& + \int_s^t \sigma(\eta_{u-r}) \,d_u(y\otimes y)_{\cdot,t},
			\end{eqnarray}
		\item for $s\in[0,r)$ and $t\in[r,T]$,
			\begin{eqnarray} \label{eq:diffdelayx0rT}
				\big( (x^r-\widehat{x}^r)\otimes y \big)_{s,t} &=& \int_s^t (y_t-y_u) b(u,x^r_u) \,du + \int_s^t \sigma(x^r_{u-r}) \,d_u(y\otimes y)_{\cdot,t} \nonumber \\
				&& - (\eta_{\cdot-r}\otimes y)_{s,r} - \int_r^t (y_t-y_u) b(u-r,\widehat{x}^r_u) \,du \nonumber \\
				&& - \int_r^t \sigma(\widehat{x}^r_{u-r}) \,d_u(y_{\cdot-r}\otimes y)_{\cdot,t} - (\eta_0-\eta_{s-r})\otimes(y_t-y_r), \nonumber
			\end{eqnarray}
		\item for $s<t\in[r,T]$,
			\begin{eqnarray} 
				&&\big( (x^r-\widehat{x}^r)\otimes y \big)_{s,t}
%\int_s^t (y_t-y_u) b(u,x^r_u) \,du + \int_s^t \sigma(x^r_{u-r}) \,d_u(y\otimes y)_{\cdot,t} \\
		%		&& - \int_s^t (y_t-y_u) b(u-r,\widehat{x}^r_u) \,du - \int_s^t \sigma(\widehat{x}^r_{u-r}) \,d_u(y_{\cdot-r}\otimes y)_{\cdot,t} \nonumber \\
				= \int_s^t (y_{u+r}-y_u) b(u,x^r_u) \,du  \label{eq:diffdelayxrT}r \\
				&& + \int_s^t \sigma(x^r_{u-r}) \,d_u \big( (y-y_{\cdot-r})\otimes y \big)_{\cdot,t} 
				 + \int_s^t \big[ \sigma(x^r_{u-r})-\sigma(\widehat{x}^r_{u-r}) \big] \,d_u(y_{\cdot-r}\otimes y)_{\cdot,t} \nonumber 
			\end{eqnarray}
	\end{itemize}

\smallskip

The following proposition gives us a result about the behavior of $(x^r-\widehat{x}^r)$ when $r$ tends to zero.
	\begin{prop}\label{prop:delay}
		Let $\beta'=\beta-\varepsilon$, where $\varepsilon>0$ is such that $\beta-2\varepsilon>0$ and $\lambda>\frac{1}{\beta-\varepsilon}-2$.  Suppose that $(x,y,x\otimes y)$, $(x^r,y,x^r\otimes y)$, $(\widehat{x}^r,y,\widehat{x}^r\otimes y)$ and $(y,y,y\otimes y)$ belong to $M^\beta_{d,m}(0,T)$. Assume that $\sigma$ and $b$ satisfy {\upshape\bfseries (H1)} and {\upshape\bfseries (H2)} respectively, and both satisfy {\upshape\bfseries (H3)}. Assume also that $\|\eta\|_{\beta(-r_0,0)}<\infty$ and $\sup_{r\leq r_0} \Phi_{\beta(0,r)}(\eta_{\cdot-r},y)<\infty$ and suppose that $\|(y-y_{\cdot-r})\otimes y\|_{2\beta'(r,T)} \rightarrow 0$ and $\|y_{\cdot-r}\otimes (y-y_{\cdot-r})\|_{2\beta'(r,T)} \rightarrow 0$ when $r$ tends to zero. Then
			\begin{eqnarray*}
				\|x^r-\widehat{x}^r\|_{\infty} &\leq& K\rho\Lambda \,r^{\beta'} \\
				\|x^r-\widehat{x}^r\|_{\beta'} &\leq& K\rho\Lambda \,r^\varepsilon \\
				\big\| (x^r-\widehat{x}^r)\otimes y \big\|_{2\beta'} &\leq& KM\rho^3\Lambda^3 \,r^\varepsilon + KM\rho^3\Lambda^2\Lambda_r
			\end{eqnarray*}
		where $K\geq1$, $M\geq1$ are constants depending on $\beta, \beta', r_0, T, \sigma, y$ and
			\begin{eqnarray*}
				\rho &=& \Big( 1 + 3\|b\|_{\infty}T^{1-\beta'} + 3\|\sigma\|_{\infty}(1+T^{\beta'}) + 2\|\sigma'\|_{\infty}(1+T^{\beta'}) + 3\|\sigma'\|_{\infty}T^{\beta'-\varepsilon} \\
				&& \hspace{5mm} + \|\sigma'\|_{\lambda}\Big( 2\sup_{r\leq r_0}\|x^r\|^\lambda_{\beta'} + \|\eta\|^\lambda_{\beta'(-r_0,0)} \Big)T^{(\lambda+1)\beta'-\varepsilon} + \|\sigma''\|_{\infty}T^{\beta'}(1+T^{\beta'}) \\
				&& \hspace{5mm} + 2\|\sigma''\|_{\lambda}\sup_{r\leq r_0}\|\widehat{x}^r\|^\lambda_{\beta'}T^{(\lambda+1)\beta'} \Big) (1+T^\varepsilon),\\
				\Lambda &=& \max\bigg( 1, \|\eta\|_{\beta(-r_0,0)}, \sup_{r\leq r_0}\Phi_{\beta'(0,r)}(\eta_{\cdot-r},y), \Phi_{\beta(0,T)}(y,y), \sup_{r\leq r_0}\Phi_{\beta'(r,T)}(y_{\cdot-r},y), \\
				&& \hspace{15mm} \sup_{r\leq r_0}\Phi_{\beta'(0,r)}(\eta_{\cdot-r},y,y), \sup_{r\leq r_0}\Phi_{\beta'(0,T)}(x^r,y), \sup_{r\leq r_0}\Phi_{\beta'(0,T)}(\widehat{x}^r,y) \bigg) \\
				&& \hspace{5mm} \times \Big( 1+\sup_{r\leq r_0}\|x^r\|_{\beta'(r)} \Big) \Big( 1+\|y\|_{\beta} \Big),\\
				\Lambda_r &=& \max\Big( 1, \sup_{r\leq r_0}\|x^r\|_{\beta'} \Big) \Big( \|(y-y_{\cdot-r})\otimes y\|_{2\beta'(r,T)} + \|y_{\cdot-r}\otimes (y-y_{\cdot-r})\|_{2\beta'(r,T)} \Big).
			\end{eqnarray*}
	\end{prop}

	\begin{rem}
	Thanks to Proposition \ref{cotesxr}, $\rho$ and $\Lambda$ are finite and, by hypothesis, $\Lambda_r$ converges to zero when $r$ tends to zero. Hence, the proposition states that
		\begin{eqnarray*}
			\|x^r-\widehat{x}^r\|_{\infty} &\xrightarrow{r\downarrow 0}& 0, \\
			\|x^r-\widehat{x}^r\|_{\beta'} &\xrightarrow{r\downarrow 0}& 0, \\
			\big\| (x^r-\widehat{x}^r)\otimes y \big\|_{2\beta'} &\xrightarrow{r\downarrow 0}& 0.
		\end{eqnarray*}
	\end{rem}

\begin{proof}
	We start studying the supremum norm.
	On one hand,  using Proposition \ref{prop1HN}, for $r\leq r_0$, we obtain
		\begin{eqnarray*}
			\|x^r-\widehat{x}^r\|_{\infty(0,r)} &\leq& \|\eta\|_{\beta'(-r,0)} r^{\beta'} + \|b\|_{\infty} r + K \|\sigma\|_{\infty} \|y\|_{\beta'} r^{\beta'} \\
			&& \hspace{5mm} + K \Phi_{\beta'(0,r)}(\eta_{\cdot-r},y) \big( \|\sigma'\|_\infty + \|\sigma'\|_\lambda \|\eta_{\cdot-r}\|^\lambda_{\beta'(0,r)} r^{\lambda\beta'} \big) r^{2\beta'} \\
			&\leq& \Big[ \|\eta\|_{\beta(-r_0,0)} T^{\varepsilon} + \|b\|_{\infty} T^{1-\beta'} + K \|\sigma\|_{\infty} \|y\|_{\beta} T^{\varepsilon} \\
			&& \hspace{5mm} + K \Phi_{\beta'(0,r)}(\eta_{\cdot-r},y) \big( \|\sigma'\|_\infty + \|\sigma'\|_\lambda \|\eta\|^\lambda_{\beta'(-r_0,0)} T^{\lambda\beta'} \big) T^{\beta'} \Big] r^{\beta'}
		\end{eqnarray*}
	where we used that $\|\eta\|_{\beta'(-r,0)} \leq \|\eta\|_{\beta(-r_0,0)}T^\varepsilon$ and $\|y\|_{\beta'} \leq \|y\|_{\beta}T^\varepsilon$. \\
	On the other hand,  using Proposition \ref{prop1HN} we obtain
		\begin{eqnarray*}
			\|x^r-\widehat{x}^r\|_{\infty(r,T)} &\leq& \|b\|_{\infty} r + K \|\sigma\|_{\infty} \|y\|_{\beta'} r^{\beta'} \\
			&& \hspace{5mm} + K \Phi_{\beta'(0,T)}(\widehat{x}^r,y) \big( \|\sigma'\|_\infty + \|\sigma'\|_\lambda \|\widehat{x}^r\|^\lambda_{\beta'(r,T)} \,r^{\lambda\beta'} \big) r^{2\beta'} \\
			&\leq& \Big[ \|b\|_{\infty} T^{1-\beta'} + K \|\sigma\|_{\infty} \|y\|_{\beta} T^{\varepsilon} \\
			&& \hspace{5mm} + K \Phi_{\beta'(0,T)}(\widehat{x}^r,y) \big( \|\sigma'\|_\infty + \|\sigma'\|_\lambda \|x^r\|^\lambda_{\beta'} \,T^{\lambda\beta'} \big) T^{\beta'} \Big] r^{\beta'}.
		\end{eqnarray*}
	Hence, we have that
		\begin{equation} \label{cotaH^1}
			\|x^r-\widehat{x}^r\|_\infty \leq K\rho\Lambda \,r^{\beta'}.
		\end{equation}

	Now we study the H\"older norms. Following the proof of Lemma \ref{lemma_cotesxr} we easily obtain that
		\begin{eqnarray} \label{cota_mult}
			\|x^r-\widehat{x}^r\|_{\beta'} \leq \|x^r-\widehat{x}^r\|_{\beta'(0,r)} + \|x^r-\widehat{x}^r\|_{\beta'(r,T)}
		\end{eqnarray}
	and
		\begin{eqnarray} \label{cota_multx}
			\|(x^r-\widehat{x}^r)\otimes y\|_{2\beta'} &\leq& \|(x^r-\widehat{x}^r)\otimes y\|_{2\beta'(0,r)} + \|(x^r-\widehat{x}^r)\otimes y\|_{2\beta'(r,T)} \nonumber \\
			&& \hspace{10mm} + \|x^r-\widehat{x}^r\|_{\beta'(0,r)}\|y\|_{\beta'}.
		\end{eqnarray}
	So we can study the H\"older norms independently in the intervals $[0,r)$ and $[r,T]$.
	We study the H\"older norm of $(x^r-\widehat{x}^r)$ .By \eqref{eq:diffdelayst0r} and Proposition \ref{prop1HN} we have
		\begin{eqnarray} \label{cota:1}
			\|x^r-\widehat{x}^r\|_{\beta'(0,r)} &\leq& \|\eta\|_{\beta'(-r,0)} + \|b\|_\infty r^{1-\beta'} + K \|\sigma\|_\infty \|y\|_{\beta'(0,r)} \nonumber\\
			&& \hspace{5mm} + K \Phi_{\beta'(0,r)}(\eta_{\cdot-r},y) \big( \|\sigma'\|_\infty + \|\sigma'\|_\lambda \|\eta_{\cdot-r}\|^\lambda_{\beta'(0,r)} r^{\lambda\beta'} \big) r^{\beta'} \nonumber\\
			&\leq& \Big[ \|\eta\|_{\beta(-r_0,0)} + \|b\|_\infty T^{1-\beta} + K \|\sigma\|_\infty \|y\|_\beta \nonumber\\
			&& \hspace{5mm} + K \Phi_{\beta'(0,r)}(\eta_{\cdot-r},y) \big( \|\sigma'\|_\infty + \|\sigma'\|_\lambda \|\eta\|^\lambda_{\beta'(-r_0,0)} T^{\lambda\beta'} \big) T^{\beta'-\varepsilon} \Big] r^{\varepsilon}.\nonumber\\
		\end{eqnarray}
	In the interval $[r,T]$, observe that
		\begin{eqnarray} \label{cota:2}
			&& \|x^r-\widehat{x}^r\|_{\beta'(r,T)} \nonumber \\
			&& \hspace{5mm} \leq \max\Bigg( \sup_{\substack{s<t\in[r,T] \\ t-s\leq r}} \frac{|(x^r-\widehat{x}^r)_t-(x^r-\widehat{x}^r)_s|}{(t-s)^{\beta'}}, \,\sup_{\substack{s<t\in[r,T] \\ t-s\geq r}} \frac{|(x^r-\widehat{x}^r)_t-(x^r-\widehat{x}^r)_s|}{(t-s)^{\beta'}} \Bigg). \nonumber\\
		\end{eqnarray}
	On one hand, by definition \eqref{eq:diffdelaystrT2} and Proposition \ref{prop1HN} we have
		\begin{eqnarray}\label{cota:3}
			&& \sup_{\substack{s<t\in[r,T] \\ t-s\leq r}} \frac{|(x^r-\widehat{x}^r)_t-(x^r-\widehat{x}^r)_s|}{(t-s)^{\beta'}} \nonumber\\
			&& \hspace{18mm} \leq 2 \|b\|_\infty r^{1-\beta'} + 2K \|\sigma\|_\infty \|y\|_{\beta} \,r^\varepsilon \nonumber\\
			&& \hspace{22mm} + 2K \Phi_{\beta'(r,T)}(\widehat{x}^r,y) \big( \|\sigma'\|_\infty + \|\sigma'\|_\lambda \|\widehat{x}^r\|^\lambda_{\beta'(r,T)} r^{\lambda\beta'} \big) r^{\beta'} \nonumber\\
			&& \hspace{18mm} \leq \Big[ 2 \|b\|_\infty T^{1-\beta} + 2K \|\sigma\|_\infty \|y\|_{\beta} \nonumber\\
			&& \hspace{22mm} + 2K \Phi_{\beta'(0,T)}(\widehat{x}^r,y) \big( \|\sigma'\|_\infty + \|\sigma'\|_\lambda \|x^r\|^\lambda_{\beta'} T^{\lambda\beta'} \big) T^{\beta'-\varepsilon} \Big] r^{\varepsilon},
		\end{eqnarray}
	where we used that $\displaystyle \sup_{\substack{s<t\in[r,T] \\ t-s\leq r}} \|y\|_{\beta'(s,t)} \leq \|y\|_{\beta} \,r^\varepsilon$. \\
	On the other hand, with a similar computation, by definition \eqref{eq:diffdelaystrT1} and Proposition \ref{prop1HN} we have
		\begin{eqnarray} \label{cota:4}
			&& \sup_{\substack{s<t\in[r,T] \\ t-s\geq r}} \frac{|(x^r-\widehat{x}^r)_t-(x^r-\widehat{x}^r)_s|}{(t-s)^{\beta'}}\nonumber \\
			&& \hspace{18mm} \leq 2 \|b\|_\infty r^{1-\beta'} + 2K \|\sigma\|_\infty \|y\|_{\beta} T^\varepsilon r^\varepsilon \nonumber\\
			&& \hspace{22mm} + 2K \Phi_{\beta'(r,T)}(\widehat{x}^r,y) \big( \|\sigma'\|_\infty + \|\sigma'\|_\lambda \|\widehat{x}^r\|^\lambda_{\beta'(r,T)} r^{\lambda\beta'} \big) r^{\beta'} \nonumber\\
			&& \hspace{18mm} \leq \Big[ 2 \|b\|_\infty T^{1-\beta} + 2K \|\sigma\|_\infty \|y\|_{\beta} T^\varepsilon \nonumber\\
			&& \hspace{22mm} + 2K \Phi_{\beta'(0,T)}(\widehat{x}^r,y) \big( \|\sigma'\|_\infty + \|\sigma'\|_\lambda \|x^r\|^\lambda_{\beta'} \,T^{\lambda\beta'} \big) T^{\beta'-\varepsilon} \Big] r^\varepsilon,
		\end{eqnarray}
	where we used that $\displaystyle \sup_{t\in[r,T]} \|y\|_{\beta'(t-r,t)} \leq \|y\|_{\beta} \,r^\varepsilon$. \\
	Then, by inequality \eqref{cota_mult}  and using \eqref{cota:1}, \eqref{cota:2}, \eqref{cota:3} and \eqref{cota:4} it follows that
		\begin{equation} \label{cotaH^2}
			\|x^r-\widehat{x}^r\|_{\beta'} \leq K\rho\Lambda \,r^\varepsilon.
		\end{equation}
	
	Finally, we study the H\"older norm $\|(x^r-\widehat{x}^r)\otimes y\|_{2\beta'}$. By definition \eqref{eq:diffdelayx0r} and Proposition \ref{prop2HN} we have
		\begin{eqnarray} \label{cota0r}
			&& \|(x^r-\widehat{x}^r)\otimes y\|_{2\beta'(0,r)} \nonumber \\
			&& \hspace{5mm} \leq \|\eta_{\cdot-r}\|_{\beta'(0,r)} \|y\|_{\beta'(0,r)} + \|y\|_{\beta'(0,r)} \|b\|_{\infty} r^{1-\beta'} + K \|\sigma\|_{\infty} \Phi_{\beta'(0,r)}(y,y) \nonumber \\
			&& \hspace{15mm} + K (\|\sigma'\|_{\infty} + \|\sigma'\|_{\lambda} \|\eta_{\cdot-r}\|_{\beta'(0,r)}^\lambda r^{\lambda\beta'}) \Phi_{\beta'(0,r)}(\eta_{\cdot-r},y,y) \,r^{\beta'} \nonumber \\
			&& \hspace{5mm} \leq \Big[ \|\eta\|_{\beta(-r_0,0)} \|y\|_{\beta} T^\varepsilon + \|y\|_{\beta} \|b\|_{\infty} T^{1-\beta'} + K \|\sigma\|_{\infty} \Phi_{\beta(0,T)}(y,y) T^\varepsilon \nonumber \\
			&& \hspace{15mm} + K (\|\sigma'\|_{\infty} + \|\sigma'\|_{\lambda} \|\eta\|_{\beta'(-r_0,0)}^\lambda T^{\lambda\beta'}) \Phi_{\beta'(0,r)}(\eta_{\cdot-r},y,y) \,T^{\beta'-\varepsilon} \Big] r^{\varepsilon} \nonumber \\
			&& \hspace{5mm} \leq K\rho\Lambda r^\varepsilon
		\end{eqnarray}
	where we used that $\Phi_{\beta'(0,r)}(y,y) \leq \Phi_{\beta(0,T)}(y,y)r^{2\varepsilon}$. \\
	Now we study the H\"older norm in the interval $[r,T]$. Let $a<b\in[r,T]$. By  \eqref{eq:diffdelayxrT}
		\begin{eqnarray} \label{A_1+A_2+A_3}
			\|(x^r-\widehat{x}^r)\otimes y\|_{2\beta'(a,b)} &\leq& \sup_{s<t\in[a,b]} \frac{\big| \int_s^t (y_{u+r}-y_u) b(u,x^r_u) \,du \big|}{(t-s)^{2\beta'}} \nonumber \\
			&& + \sup_{s<t\in[a,b]} \frac{\big| \int_s^t \sigma(x^r_{u-r}) \,d_u \big( (y-y_{\cdot-r})\otimes y \big)_{\cdot,t} \big|}{(t-s)^{2\beta'}} \nonumber \\
			&& + \sup_{s<t\in[a,b]} \frac{\big| \int_s^t \big[ \sigma(x^r_{u-r})-\sigma(\widehat{x}^r_{u-r}) \big] \,d_u(y_{\cdot-r}\otimes y)_{\cdot,t} \big|}{(t-s)^{2\beta'}} \nonumber \\
			&=& A_1 + A_2 + A_3.
		\end{eqnarray}
	It is easy to see that    
		\begin{eqnarray} \label{A_1}
			A_1 \leq \|y\|_{\beta} \,\|b\|_{\infty} T^{1-\beta'} \leq K\rho\Lambda r^{\varepsilon}.
		\end{eqnarray}
	By Proposition \ref{prop2HN} we have
		\begin{eqnarray} \label{cota:A2}
			A_2 \hspace{-2mm} &\leq& \hspace{-2mm} K \|\sigma\|_{\infty} \,\Phi_{\beta'(a,b)}(y-y_{\cdot-r},y) \nonumber \\
			&& \hspace{-2mm}+ \,K \,\left( \|\sigma'\|_\infty + \|\sigma'\|_\lambda \|\widehat{x}^r\|^\lambda_{\beta'(a,b)} \,T^{\lambda\beta'} \right) \Phi_{\beta'(a,b)}(\widehat{x}^r,y-y_{\cdot-r},y) \,T^{\beta'} \nonumber\\
			&=&\hspace{-2mm} K \|y\|_{\beta'}\left( \|\sigma\|_{\infty} + \left( \|\sigma'\|_\infty + \|\sigma'\|_\lambda \|\widehat{x}^r\|^\lambda_{\beta'} \,T^{\lambda\beta'} \right) \|\widehat{x}^r\|_{\beta'} T^{\beta'}\right)\|y-y_{\cdot-r}\|_{\beta'(r,T)} \nonumber \\
			&&  \hspace{-2mm}+ K \left( \|\sigma\|_{\infty} + \left( \|\sigma'\|_\infty + \|\sigma'\|_\lambda \|\widehat{x}^r\|^\lambda_{\beta'} \,T^{\lambda\beta'} \right) \|\widehat{x}^r\|_{\beta'} T^{\beta'}\right) \|(y-y_{\cdot-r})\otimes y\|_{2\beta'(r,T)} \nonumber\\
			&& \hspace{-2mm} + \,K \|y\|_{\beta'} T^{\beta'}\left( \|\sigma'\|_\infty + \|\sigma'\|_\lambda \|\widehat{x}^r\|^\lambda_{\beta'} \,T^{\lambda\beta'} \right) \|\widehat{x}^r\otimes (y-y_{\cdot-r})\|_{2\beta'(a,b)}.
		\end{eqnarray}
	Now we will estimate the norm $\|\widehat{x}^r\otimes (y-y_{\cdot-r})\|_{2\beta'(a,b)}$. For $s<t\in[a,b]$,
		\begin{eqnarray*}
			\big( \widehat{x}^r\otimes (y-y_{\cdot-r}) \big)_{s,t}
%% &=& \int_s^t (y_t-y_u) b(u-r,\widehat{x}^r_u) \,du - \int_s^t (y_{t-r}-y_{u-r}) b(u-r,\widehat{x}^r_u) \,du \\
%			&& + \int_s^t \sigma(\widehat{x}^r_{u-r}) \,d_u(y_{\cdot-r}\otimes y)_{\cdot,t} - \int_s^t \sigma(\widehat{x}^r_{u-r}) \,d_u(y_{\cdot-r}\otimes y_{\cdot-r})_{\cdot,t} \\
			&=& \int_s^t (y_t-y_{t-r}-y_u+y_{u-r}) b(u-r,\widehat{x}^r_u) \,du \\
			&& + \int_s^t \sigma(\widehat{x}^r_{u-r}) \,d_u \big( y_{\cdot-r}\otimes (y-y_{\cdot-r}) \big)_{\cdot,t}.
		\end{eqnarray*}
	So by Proposition \ref{prop2HN} and Lemma \ref{lemma_yyr} we have
		\begin{eqnarray} \label{ineq_y2}
			&& \|\widehat{x}^r\otimes(y-y_{\cdot-r})\|_{2\beta'(a,b)} \nonumber \\
%			&& \hspace{10mm} \leq 2 \|b\|_{\infty} \|y\|_{\beta} T^{1-\beta'} r^\varepsilon + K \|\sigma\|_{\infty} \,\Phi_{\beta'(a,b)}(y_{\cdot-r},y-y_{\cdot-r}) \nonumber \\
%			&& \hspace{15mm} + K \big( \|\sigma'\|_\infty + \|\sigma'\|_\lambda \|\widehat{x}^r\|^\lambda_{\beta'(a,b)} T^{\lambda\beta'} \big) \Phi_{\beta'(a,b)}(\widehat{x}^r,y_{\cdot-r},y-y_{\cdot-r}) T^{\beta'} \nonumber \\
			&& \hspace{10mm} \le K \left[ \big( \|\sigma'\|_\infty + \|\sigma'\|_\lambda \|\widehat{x}^r\|^\lambda_{\beta'} \,T^{\lambda\beta'} \big) \big( \|\widehat{x}^r\|_{\beta'} \|y\|_{\beta'} + \|\widehat{x}^r\otimes y\|_{2\beta'} \big) T^{\beta'} \right. \nonumber \\
			&& \hspace{25mm} \left.+ \|b\|_\infty T^{1-\beta'}+\|\sigma\|_\infty \|y\|_{\beta'}\right] \|y\|_{\beta} \,r^\varepsilon \nonumber\\
			&& \hspace{15mm} + K \left[ \|\sigma\|_\infty + \big( \|\sigma'\|_\infty\hspace{-1mm} + \|\sigma'\|_\lambda \|\widehat{x}^r\|^\lambda_{\beta'} \,T^{\lambda\beta'} \big) \|\widehat{x}^r\|_{\beta'}T^{\beta'}\right] \nonumber \\
			&& \hspace{25mm} \times \|y_{\cdot-r}\otimes (y-y_{\cdot-r})\|_{2\beta'(r,T)}. \nonumber 
		\end{eqnarray}
	Putting together \eqref{cota:A2} and \eqref{ineq_y2} and inequality \eqref{yyr_beta'} we get
		\begin{eqnarray} \label{A_2}
			A_2 &\leq& K\rho\Lambda \,r^\varepsilon + K\rho^2\Lambda^2 \,r^\varepsilon + K\rho^2\Lambda\Lambda_r + K\rho\Lambda_r \nonumber \\
			&\leq& K\rho^2\Lambda^2 \,r^\varepsilon + K\rho^2\Lambda\Lambda_r,
		\end{eqnarray}
	where we used that $1\leq\rho\leq\rho^2$ and $1\leq\Lambda\leq\Lambda^2$.
	
	Finally, by Proposition \ref{prop:normbetarfx} and inequalities \eqref{cotaH^1} and \eqref{cotaH^2} we have
		\begin{eqnarray} \label{eqC^5}
			A_3 &\leq&
% G^4_{\beta'(a,b)}(\sigma,\widehat{x}^r,\widehat{x}^r_{\cdot-r},y_{\cdot-r},y) (b-a)^{\beta'} \|\widehat{x}^r-\widehat{x}^r_{\cdot-r}\|_{\infty(a,b)} \nonumber \\
%			&& + G^5_{\beta'(a,b)}(\sigma,\widehat{x}^r,\widehat{x}^r_{\cdot-r},y_{\cdot-r},y) (b-a)^{\beta'} \|\widehat{x}^r-\widehat{x}^r_{\cdot-r}\|_{\beta'(a,b)} \nonumber \\
%			&& + G^6_{\beta'(a,b)}(\sigma,\widehat{x}^r_{\cdot-r},y) (b-a)^{\beta'} \|(\widehat{x}^r-\widehat{x}^r_{\cdot-r})\otimes y_{\cdot-r}\|_{2\beta'(a,b)} \nonumber \\
		%	&\leq&
 G^4_{\beta'(r,T)}(\sigma,\widehat{x}^r,\widehat{x}^r_{\cdot-r},y_{\cdot-r},y) \,T^{\beta'} \|x^r-\widehat{x}^r\|_{\infty} \nonumber \\
			&& + G^5_{\beta'(r,T)}(\sigma,\widehat{x}^r,\widehat{x}^r_{\cdot-r},y_{\cdot-r},y) \,T^{\beta'} \|x^r-\widehat{x}^r\|_{\beta'} \nonumber \\
			&& + G^6_{\beta'(r,T)}(\sigma,\widehat{x}^r_{\cdot-r},y) (b-a)^{\beta'} \|(x^r-\widehat{x}^r)\otimes y\|_{2\beta'(a-r,b-r)} \nonumber \\
			&\leq& K\rho^2\Lambda^2 \,r^\varepsilon + G^6_{\beta'(r,T)}(\sigma,\widehat{x}^r_{\cdot-r},y) (b-a)^{\beta'} \|(x^r-\widehat{x}^r)\otimes y\|_{2\beta'(a-r,b-r)} \nonumber 
		\end{eqnarray}
	where we used that $G^i_{\beta'(r,T)}(\sigma,\widehat{x}^r,\widehat{x}^r_{\cdot-r},y_{\cdot-r},y) \,T^{\beta'} \leq K\rho\Lambda$ for $i=4,5$.

	\smallskip
	Applying the multiplicative property, it is easy to see that
		\begin{eqnarray*}
			&& \|(x^r-\widehat{x}^r)\otimes y\|_{2\beta'(a-r,b-r)} \\
			&& \hspace{10mm} \leq \|(x^r-\widehat{x}^r)\otimes y\|_{2\beta'(a-r,a)} + \|(x^r-\widehat{x}^r)\otimes y\|_{2\beta'(a,b)} + \|x^r-\widehat{x}^r\|_{\beta'} \|y\|_{\beta'}.
		\end{eqnarray*}
	On one hand, by \eqref{cotaH^2}
		\begin{eqnarray*}
			\|x^r-\widehat{x}^r\|_{\beta'} \|y\|_{\beta'} \leq \|x^r-\widehat{x}^r\|_{\beta'} \|y\|_{\beta} T^\varepsilon \leq K\rho^2\Lambda^2 \,r^\varepsilon.
		\end{eqnarray*}
	On the other hand, we have that
		\begin{eqnarray*}
			\|(x^r-\widehat{x}^r)\otimes y\|_{2\beta'(a-r,a)} \leq K\rho^2\Lambda^2 \,r^\varepsilon + K\rho^2\Lambda\Lambda_r,
		\end{eqnarray*}
	where the result is obtained considering separately the two cases $a\in[r,2r)$ and $a\in[2r,T]$ and applying multiplicative property, inequalities \eqref{cota0r}, \eqref{A_1+A_2+A_3}, \eqref{A_1}, \eqref{A_2}, \eqref{eqC^5} and $G^6_{\beta'(r,T)}(\sigma,\widehat{x}^r_{\cdot-r},y) \,T^{\beta'} \leq K\rho\Lambda$.
	Therefore,
		\begin{eqnarray*}
			\|(x^r-\widehat{x}^r)\otimes y\|_{2\beta'(a-r,b-r)} \leq K\rho^2\Lambda^2 \,r^\varepsilon + K\rho^2\Lambda\Lambda_r + \|(x^r-\widehat{x}^r)\otimes y\|_{2\beta'(a,b)}.
		\end{eqnarray*}
	Then it follows that
		\begin{eqnarray} 
			A_3 \leq G^6_{\beta'(r,T)}(\sigma,\widehat{x}^r_{\cdot-r},y) (b-a)^{\beta'} \|(x^r-\widehat{x}^r)\otimes y\|_{2\beta'(a,b)} + K\rho^3\Lambda^3 \,r^\varepsilon + K\rho^3\Lambda^2\Lambda_r, \label{A_3}
		\end{eqnarray}
	where we used again that $G^6_{\beta'(r,T)}(\sigma,\widehat{x}^r_{\cdot-r},y) \,T^{\beta'} \leq K\rho\Lambda$.
	From inequalities \eqref{A_1+A_2+A_3}, \eqref{A_1}, \eqref{A_2} and \eqref{A_3} we have that
		\begin{eqnarray*}
			\|(x^r-\widehat{x}^r)\otimes y\|_{2\beta'(a,b)} &\leq& G^6_{\beta'(r,T)}(\sigma,\widehat{x}^r_{\cdot-r},y) (b-a)^{\beta'} \|(x^r-\widehat{x}^r)\otimes y\|_{2\beta'(a,b)} \\
			&& \hspace{10mm} + K\rho^3\Lambda^3 \,r^\varepsilon + K\rho^3\Lambda^2\Lambda_r,
		\end{eqnarray*}
	where we used that $\rho^n\leq\rho^{n+1}$ and $\Lambda^n\leq\Lambda^{n+1}$ for any $n\in\N$. \\
	Set now
		\begin{equation}
			\widetilde{\Delta} := \Big( 2 \,\sup_{r\leq r_0} G^6_{\beta'(r,T)}(\sigma,\widehat{x}^r_{\cdot-r},y) \Big)^{-\frac{1}{\beta'}}.\nonumber
		\end{equation}
	Observe that, if $a,b$ are such that $(b-a)\leq\widetilde{\Delta}$, then
		\begin{eqnarray} \label{cota10}
			\|(x^r-\widehat{x}^r)\otimes y\|_{2\beta'(a,b)} \leq K\rho^3\Lambda^3 \,r^\varepsilon + K\rho^3\Lambda^2\Lambda_r.
		\end{eqnarray}
	
	Now consider a partition $r=t_0<\dots<t_M=T$ such that $(t_{i+1}-t_i)\leq\widetilde{\Delta}$ for $i=0\dots,M-1$. Then, using the multiplicative property iteratively, we have
		\begin{eqnarray*}
			\|(x^r-\widehat{x}^r)\otimes y\|_{2\beta'(r,T)} &\leq& \sum_{i=0}^{M-1} \|(x^r-\widehat{x}^r)\otimes y\|_{2\beta'(t_i,t_{i+1})} + (M-1) \|x^r-\widehat{x}^r\|_{\beta'} \|y\|_{\beta'}.
		\end{eqnarray*}
	Applying \eqref{cotaH^2} and \eqref{cota10}, we obtain
		\begin{eqnarray} \label{cotarT}
			\|(x^r-\widehat{x}^r)\otimes y\|_{2\beta'(r,T)} &\leq& KM\rho^3\Lambda^3 \,r^\varepsilon + KM\rho^3\Lambda^2\Lambda_r + K(M-1)\rho^2\Lambda^2 \,r^\varepsilon \nonumber \\
			&\leq& KM\rho^3\Lambda^3 \,r^\varepsilon + KM\rho^3\Lambda^2\Lambda_r.
		\end{eqnarray}
	
	Finally, putting together \eqref{cota_multx}, \eqref{cotaH^2}, \eqref{cota0r} and \eqref{cotarT} we have that
		\begin{eqnarray} %\label{cotaH^3}
			\|(x^r-\widehat{x}^r)\otimes y\|_{2\beta'} \leq KM\rho^3\Lambda^3 \,r^\varepsilon + KM\rho^3\Lambda^2\Lambda_r.\nonumber
		\end{eqnarray}
	So the proof is complete.
\end{proof}

\vspace{5mm}

The following definitions will be useful in the next results:
	\begin{eqnarray*}
		\overline{G}^i_{\beta'} &:=& \sup_{r\leq r_0} G^i_{\beta'(0,T)}(\sigma,x,x^r,y) \qquad i=1,2 \label{supG12} \\
		\overline{G}^3_{\beta'} &:=& \sup_{r\leq r_0} G^3_{\beta'(0,T)}(\sigma,x^r) \label{supG3} \\
		\overline{G}^j_{\beta'} &:=& \sup_{r\leq r_0} G^j_{\beta'(0,T)}(\sigma,x,x^r,y,y) \qquad j=4,5 \label{supG45} \\
		\overline{G}^6_{\beta'} &:=& \sup_{r\leq r_0} G^6_{\beta'(0,T)}(\sigma,x^r,y). \label{supG6}
	\end{eqnarray*}

The following result gives as a bound for $\big\| (x-x^r)\otimes y \big\|_{2\beta'(a,b)}$ when the interval $(a,b)$ is sufficiently small.
Define $\Delta_{\beta'}^1$ as follows:
	\begin{equation*}
		\Delta_{\beta'}^1 = \big( 2 \overline{G}^6_{\beta'} \big)^{-\frac{1}{\beta'}}.
	\end{equation*}

\hspace{5mm}

We state the following proposition:
	\begin{prop}\label{prop:norm2betaotimes}
	Suppose that $(x,y,x\otimes y)$, $(x^r,y,x^r\otimes y)$ and $(\widehat{x}^r,y,\widehat{x}^r\otimes y)$ belong to $M^\beta_{d,m}(0,T)$, $(y,y,y\otimes y)$ belongs to $M^\beta_{d,m}(0,T)$ and $(y_{\cdot-r},y,y_{\cdot-r}\otimes y)$ belongs to $M^\beta_{d,m}(r,T)$. Assume that $\sigma$ and $b$ satisfy {\upshape\bfseries (H1)} and {\upshape\bfseries (H2)} respectively. Then, for all $0\leq a<b\leq T$ such that $(b-a)\leq\Delta_{\beta'}^1$,
		\begin{eqnarray*}
			&& \big\| (x-x^r)\otimes y \big\|_{2\beta'(a,b)} \\
			&& \hspace{10mm} \leq 2 \big[ L_N \|y\|_{\beta'} (b-a)^{1-2\beta'} + G^4_{\beta'(a,b)}(\sigma,x,x^r,y,y) \big] (b-a)^{\beta'} \,\|x-x^r\|_{\infty(a,b)} \\
			&& \hspace{15mm} + \,2 \,G^5_{\beta'(a,b)}(\sigma,x,x^r,y,y) \,(b-a)^{\beta'} \,\|x-x^r\|_{\beta'(a,b)} \\
			&& \hspace{15mm} + \,2 K\rho\Lambda \,G^4_{\beta'(a,b)}(\sigma,x^r,\widehat{x}^r,y,y) (b-a)^{\beta'} r^{\beta'} \\
			&& \hspace{15mm} + \,2 K\rho\Lambda \left[\,G^5_{\beta'(a,b)}(\sigma,x^r,\widehat{x}^r,y,y)+M\rho^2\Lambda^2G^6_{\beta'(a,b)}(\sigma,\widehat{x}^r,y)\right] (b-a)^{\beta'} r^{\varepsilon} \\
			&& \hspace{15mm} + \,2 KM\rho^3\Lambda^2 \,G^6_{\beta'(a,b)}(\sigma,\widehat{x}^r,y) \,(b-a)^{\beta'}\Lambda_r.
		\end{eqnarray*}
	\end{prop}

\begin{proof}
	The proposition is proved applying first Proposition \ref{prop:normbetab}, Proposition \ref{prop:normbetafx} and Proposition \ref{prop:normbetarfx} to definition \eqref{def:differenceotimes} and then Proposition \ref{prop:delay}, and observing that for $a<b$ such that $(b-a)\leq\Delta_{\beta'}^1$
		\[ G^6_{\beta'(a,b)}(\sigma,x^r,y) \,(b-a)^{\beta'} \leq \frac12. \]
\end{proof}

\section{Proof of the main theorem} \label{dSDE_Proof}

{\it Proof of Theorem  \ref{thm_determ}}:
We start studying $\lim_{r\rightarrow0} \|x-x^r\|_\infty$.
As in Lemma \ref{lemma_cotesxr}, we can study separately the intervals $[0,r)$ and $(r,T)$.
	
First we study the norm in the interval $[0,r)$. We apply Proposition \ref{prop:normbetab} and Proposition \ref{prop:normbetaf} to  \eqref{def:difference2} and we obtain
	\begin{eqnarray*}
		\|x-x^r\|_{\beta(0,r)} &\leq& L_N \,r^{1-\beta} \|x-x^r\|_{\infty(0,r)} + G^1_{\beta(0,r)}(\sigma,x,\eta_{\cdot-r},y) \,r^{\beta} \|x-\eta_{\cdot-r}\|_{\infty(0,r)} \\
		&& \hspace{5mm} + \,G^2_{\beta(0,r)}(\sigma,x,\eta_{\cdot-r},y) \,r^{\beta} \|x-\eta_{\cdot-r}\|_{\beta(0,r)} \\
		&& \hspace{5mm} + \,G^3_{\beta(0,r)}(\sigma,\eta_{\cdot-r}) \,r^{\beta} \|(x-\eta_{\cdot-r})\otimes y\|_{2\beta(0,r)}.
	\end{eqnarray*}
Using that the supremum norm of $x$ is bounded and the bound does not depend on $r$, we  see that $\sup_{r\leq r_0} G^i_{\beta(0,r)}(\sigma,x,\eta_{\cdot-r},y)<\infty$  $i=1,2$ and $\sup_{r\leq r_0} G^3_{\beta(0,r)}(\sigma,\eta_{\cdot-r})<\infty$. So last expression clearly goes to zero when $r$ tends to zero. 

\smallskip
Now we work on the interval $[r,T]$. Let $r\leq a<b\leq T$. Applying Proposition \ref{prop:normbetab}, Proposition \ref{prop:normbetaf}, Proposition \ref{prop:normbetarf} and Proposition \ref{prop:delay}, we obtain
	\begin{eqnarray*}
		\|x-x^r\|_{\beta'(a,b)} &\leq& %\big[ L_N (b-a)^{1-2\beta'} + G^1_{\beta'(a,b)}(\sigma,x,x^r,y) \big] \,(b-a)^{\beta'} \|x-x^r\|_{\infty(a,b)} \\
%		&& \hspace{5mm} + \,G^2_{\beta'(a,b)}(\sigma,x,x^r,y) \,(b-a)^{\beta'} \|x-x^r\|_{\beta'(a,b)} \\
%		&& \hspace{5mm} + \,G^3_{\beta'(a,b)}(\sigma,x^r) \,(b-a)^{\beta'} \|(x-x^r)\otimes y\|_{2\beta'(a,b)} \\
	%	&& \hspace{5mm} + \,G^1_{\beta'(a,b)}(\sigma,x^r,\widehat{x}^r,y) \,(b-a)^{\beta'} \|x^r-\widehat{x}^%r\|_{\infty(a,b)}\\
	%	&& \hspace{5mm} + \,G^2_{\beta'(a,b)}(\sigma,x^r,\widehat{x}^r,y) \,(b-a)^{\beta'} \|x^r-\widehat{x}^r\|_{\beta'(a,b)}\\
	%	&& \hspace{5mm} + \,G^3_{\beta'(a,b)}(\sigma,\widehat{x}^r) \,(b-a)^{\beta'} \|(x^r-\widehat{x}^r)\otimes y\|_{2\beta'(a,b)} \\
	%	&\leq&
 \big[ L_N (b-a)^{1-2\beta'} + G^1_{\beta'(a,b)}(\sigma,x,x^r,y) \big] \,(b-a)^{\beta'} \|x-x^r\|_{\infty(a,b)} \\
		&& \hspace{5mm} + \,G^2_{\beta'(a,b)}(\sigma,x,x^r,y) \,(b-a)^{\beta'} \|x-x^r\|_{\beta'(a,b)} \\
		&& \hspace{5mm} + \,G^3_{\beta'(a,b)}(\sigma,x^r) \,(b-a)^{\beta'} \|(x-x^r)\otimes y\|_{2\beta'(a,b)} \\
		&& \hspace{5mm} + K\rho\Lambda \,G^1_{\beta'(a,b)}(\sigma,x^r,\widehat{x}^r,y) (b-a)^{\beta'} r^{\beta'} \\
		&& \hspace{5mm} + \big[ K\rho\Lambda \,G^2_{\beta'(a,b)}(\sigma,x^r,\widehat{x}^r,y) + K\rho^3\Lambda^3 \,G^3_{\beta'(a,b)}(\sigma,\widehat{x}^r) \big] \,(b-a)^{\beta'} r^\varepsilon \\
		&& \hspace{5mm} + K\rho^3\Lambda^2\Lambda_r \,G^3_{\beta'(a,b)}(\sigma,\widehat{x}^r) \,(b-a)^{\beta'}.
	\end{eqnarray*}

Set
	\begin{eqnarray} %\label{H}
		H_r &:=& K\rho\Lambda \big( G^1_{\beta'(0,T)}(\sigma,x^r,\widehat{x}^r,y) + 2G^3_{\beta'(0,T)}(\sigma,x^r) G^4_{\beta'(0,T)}(\sigma,x^r,\widehat{x}^r,y,y) \,T^{\beta'} \big) \,T^{\beta'} r^{\beta'} \nonumber \\
		&& + \Big[ K\rho\Lambda \big( G^2_{\beta'(0,T)}(\sigma,x^r,\widehat{x}^r,y) + 2G^3_{\beta'(0,T)}(\sigma,x^r) G^5_{\beta'(0,T)}(\sigma,x^r,\widehat{x}^r,y,y) \,T^{\beta'} \big) \nonumber \\
		&& \hspace{5mm} + K\rho^3\Lambda^3 \big( G^3_{\beta'(0,T)}(\sigma,\widehat{x}^r) + 2G^3_{\beta'(0,T)}(\sigma,x^r) G^6_{\beta'(0,T)}(\sigma,\widehat{x}^r,y) T^{\beta'} \big) \Big] \,T^{\beta'} r^\varepsilon \nonumber \\
		&& + K\rho^3\Lambda^2 \big( G^3_{\beta'(0,T)}(\sigma,\widehat{x}^r) + 2G^3_{\beta'(0,T)}(\sigma,x^r) G^6_{\beta'(0,T)}(\sigma,\widehat{x}^r,y) T^{\beta'} \big) \Big] \,T^{\beta'}\Lambda_r. \nonumber
	\end{eqnarray}
Observe that $H_r$ converges to zero when $r$ tends to zero.

Then we take $a$ and $b$ such that 
	\begin{equation} \label{condi1}
		(b-a)\leq\Delta_{\beta'}^1
	\end{equation}
and apply Proposition \ref{prop:norm2betaotimes} to get that
	\begin{eqnarray*}
		&& \|x-x^r\|_{\beta'(a,b)} \\
		&& \hspace{3.5mm} \leq \Big[ L_N \big( 2\|y\|_{\beta'} G^3_{\beta'(a,b)}(\sigma,x^r) (b-a)^{\beta'} +1 \big) (b-a)^{1-2\beta'} \\
		&& \hspace{7mm} + G^1_{\beta'(a,b)}(\sigma,x,x^r,y) + 2G^3_{\beta'(a,b)}(\sigma,x^r) G^4_{\beta'(a,b)}(\sigma,x,x^r,y,y) (b-a)^{\beta'} \Big] \\
		&& \hspace{20mm} \times \,(b-a)^{\beta'} \|x-x^r\|_{\infty(a,b)} \\
		&& \hspace{7mm} + \Big[ \,G^2_{\beta'(a,b)}(\sigma,x,x^r,y) + 2G^3_{\beta'(a,b)}(\sigma,x^r) G^5_{\beta'(a,b)}(\sigma,x,x^r,y,y) (b-a)^{\beta'} \Big] \\
		&& \hspace{20mm} \times \,(b-a)^{\beta'} \|x-x^r\|_{\beta'(a,b)} \\
		&& \hspace{7mm} + K\rho\Lambda \big( G^1_{\beta'(a,b)}(\sigma,x^r,\widehat{x}^r,y) + 2G^3_{\beta'(a,b)}(\sigma,x^r) G^4_{\beta'(a,b)}(\sigma,x^r,\widehat{x}^r,y,y) (b-a)^{\beta'} \big) \\
		&& \hspace{20mm} \times \,(b-a)^{\beta'} r^{\beta'} \\
		&& \hspace{7mm} + \Big[ K\rho\Lambda \big( G^2_{\beta'(a,b)}(\sigma,x^r,\widehat{x}^r,y) + 2G^3_{\beta'(a,b)}(\sigma,x^r) G^5_{\beta'(a,b)}(\sigma,x^r,\widehat{x}^r,y,y) (b-a)^{\beta'} \big) \\
		&& \hspace{15mm} + K\rho^3\Lambda^3 \big( G^3_{\beta'(a,b)}(\sigma,\widehat{x}^r) + 2G^3_{\beta'(a,b)}(\sigma,x^r) G^6_{\beta'(a,b)}(\sigma,\widehat{x}^r,y) (b-a)^{\beta'} \big) \Big] \\
		&& \hspace{20mm} \times \,(b-a)^{\beta'} r^\varepsilon \\
		&& \hspace{7mm} + K\rho^3\Lambda^2 \big( G^3_{\beta'(a,b)}(\sigma,\widehat{x}^r) + 2G^3_{\beta'(a,b)}(\sigma,x^r) G^6_{\beta'(a,b)}(\sigma,\widehat{x}^r,y) (b-a)^{\beta'} \big)(b-a)^{\beta'}\Lambda_r\\
		&& \hspace{7mm} \leq \Big[ L_N \big( 2\|y\|_{\beta'} G^3_{\beta'(a,b)}(\sigma,x^r) T^{\beta'} +1 \big) T^{1-2\beta'} 
		+ G^1_{\beta'(a,b)}(\sigma,x,x^r,y) \\ && \hspace{5mm} + 2G^3_{\beta'(a,b)}(\sigma,x^r) G^4_{\beta'(a,b)}(\sigma,x,x^r,y,y) T^{\beta'} \Big]
		 (b-a)^{\beta'} \|x-x^r\|_{\infty(a,b)} \\
		&& \hspace{7mm}+ \big[ \,G^2_{\beta'(a,b)}(\sigma,x,x^r,y) + 2G^3_{\beta'(a,b)}(\sigma,x^r) G^5_{\beta'(a,b)}(\sigma,x,x^r,y,y) T^{\beta'} \big] \\
		&& \hspace{20mm} \times \, (b-a)^{\beta'} \|x-x^r\|_{\beta'(a,b)}   
		 + H_r. \\
	\end{eqnarray*}
If we take now $a$ and $b$ such that
	\begin{eqnarray} \label{condi2}
		\big[ \,G^2_{\beta'(a,b)}(\sigma,x,x^r,y) + 2G^3_{\beta'(a,b)}(\sigma,x^r) G^5_{\beta'(a,b)}(\sigma,x,x^r,y,y) T^{\beta'} \big] \,(b-a)^{\beta'} \leq \frac12,
	\end{eqnarray}
we get
	\begin{eqnarray} \label{eq:bddbeta}
		& &\|x-x^r\|_{\beta'(a,b)} \leq 2 \Big[ L_N \big( 2\|y\|_{\beta'} G^3_{\beta'(a,b)}(\sigma,x^r) T^{\beta'} +1 \big) T^{1-2\beta'}+ G^1_{\beta'(a,b)}(\sigma,x,x^r,y)  \nonumber \\
		&& \hspace{2mm} + 2G^3_{\beta'(a,b)}(\sigma,x^r) G^4_{\beta'(a,b)}(\sigma,x,x^r,y,y) T^{\beta'} \Big] (b-a)^{\beta'} \|x-x^r\|_{\infty(a,b)} + 2H_r.
	\end{eqnarray}
On the other hand,
	\begin{eqnarray*}
		\|x-x^r\|_{\infty(a,b)} &\leq& |x_a-x^r_a| + (b-a)^{\beta'} \|x-x^r\|_{\beta'(a,b)},
	\end{eqnarray*}
and replacing in \eqref{eq:bddbeta} we obtain
	\begin{eqnarray}
		\|x-x^r\|_{\infty(a,b)} &\leq& |x_a-x^r_a| + 2 \Big[ L_N \big( 2\|y\|_{\beta'} G^3_{\beta'(a,b)}(\sigma,x^r) T^{\beta'} +1 \big) T^{1-2\beta'} \nonumber \\
		&& \hspace{5mm} + G^1_{\beta'(a,b)}(\sigma,x,x^r,y) + 2G^3_{\beta'(a,b)}(\sigma,x^r) G^4_{\beta'(a,b)}(\sigma,x,x^r,y,y) T^{\beta'} \Big] \nonumber \\
		&& \hspace{15mm} \times \,(b-a)^{2\beta'} \|x-x^r\|_{\infty(a,b)} + 2T^{\beta'} H_r.\nonumber
	\end{eqnarray}
If we take now  $a$ and $b$ such that	
	\begin{eqnarray} \label{condi3}
		&& 2 \Big[ L_N \big( 2\|y\|_{\beta'} G^3_{\beta'(a,b)}(\sigma,x^r) T^{\beta'} +1 \big) T^{1-2\beta'} + G^1_{\beta'(a,b)}(\sigma,x,x^r,y) \nonumber \\
		&& \hspace{5mm} + 2G^3_{\beta'(a,b)}(\sigma,x^r) G^4_{\beta'(a,b)}(\sigma,x,x^r,y,y) T^{\beta'} \Big] T^{\beta'} (b-a)^{\beta'} \leq \frac12 
	\end{eqnarray}	
we obtain
	\begin{eqnarray*} 
		\|x-x^r\|_{\infty(a,b)} &\leq& 2 |x_a-x^r_a| + 4T^{\beta'} H_r,
	\end{eqnarray*}
and hence
	\begin{eqnarray} \label{cota(0,b)}
		\sup_{0\leq t\leq b} |x_t-x^r_t| &\leq& 2 \sup_{0\leq t\leq a} |x_t-x^r_t| + 4T^{\beta'} H_r.
	\end{eqnarray}
We define $\Delta_{\beta'}$ such that all $a,b$ with $(b-a)\leq\Delta_{\beta'}$ fulfill the following conditions \eqref{condi1}, \eqref{condi2} and \eqref{condi3}, that is	
	\begin{eqnarray}
		\Delta_{\beta'} &:=& \bigg( 16 L_N T^{1-\beta'} + 16 \,\overline{G}^1_{\beta'} \,T^{\beta'} + 4 \,\overline{G}^2_{\beta'} \nonumber \\
		&& \hspace{5mm} + 8 \,\overline{G}^3_{\beta'} \big[ 4L_N \|y\|_{\beta'} T + 4 \,\overline{G}^4_{\beta'} \,T^{2\beta'} + \,\overline{G}^5_{\beta'} \,T^{\beta'} \big] + 2 \,\overline{G}^6_{\beta'} \bigg)^{-\frac{1}{\beta'}}. \nonumber 
	\end{eqnarray}
Then, it is clear that \eqref{cota(0,b)} holds for all $a$ and $b$ such that $b-a\leq \Delta_{\beta'}$.

\smallskip
Now, we take a partition $0=t_0<t_1<\cdots<t_M=T$ of the interval $[0,T]$ such that $(t_{i+1}-t_i)\leq\Delta_{\beta'}$. Then,
	\begin{eqnarray}
		\sup_{0\leq t\leq t_M=T} |x_t-x^r_t| &\leq& 2 \sup_{0\leq t\leq t_{M-1}} |x_t-x^r_t| + 4T^{\beta'} H_r.\nonumber
	\end{eqnarray}
Repeating the process $M$ times we obtain
	\begin{eqnarray*}
		\sup_{0\leq t\leq T} |x_t-x^r_t| &\leq& 2^M |x_0-x^r_0| + \bigg(\sum_{k=0}^{M-1}2^k\bigg) 4T^{\beta'} H_r = 4(2^M-1)T^{\beta'} H_r
	\end{eqnarray*}
that clearly converges to zero when $r$ tends to zero.

Following the same arguments it follows that	\[ \lim_{r\rightarrow0} \|(x\otimes y)-(x^r\otimes y)\|_{\infty} = \lim_{r\rightarrow0} \|(x-x^r)\otimes y\|_{\infty} \].
\qed

\section{Stochastic case}

In this section we apply the results obtained in the deterministic case to the case of the Brownian motion in order to get convergence of stochastic differential equations driven by Brownian motion. 

Suppose that $B=\{B_t=(B_t^1,B_t^2,\dots,B_t^m), \,t\geq0\}$ is a  $m$-dimensional Brownian motion. Fix a time interval $[0,T]$. Then, for $s,t\in[0,T]$ and $i,j\in\{1,\dots,m\}$, we consider the following tensor products:
	\begin{equation*}
		(B^i\otimes B^j)_{s,t} := \int_s^t (B^i_u-B^i_s) \,d^\circ B^j_u - \frac12 (t-s) \,\1_{\{i=j\}}
	\end{equation*}
and
	\begin{equation*}
		(B^i\otimes B^j_{\cdot-r})_{s,t} := \int_s^t (B^i_u-B^i_s) \,d^\circ B^j_{u-r},
	\end{equation*}
where the stochastic integral is a Stratonovich integral (see Russo and Vallois \cite{RV}). In \cite{NNT}  we see that we can choose a version $(B\otimes B_{\cdot-r})_{s,t}$ in such a way that  $(B_{\cdot-r},B,B\otimes B_{\cdot-r})$ constitutes a $\beta$-H\"older continuous multiplicative functional, for a fixed $\beta\in(\frac13,\frac12)$. On the other hand, from Hu and Nualart  \cite{HN}  it follows that  $(B\otimes B)_{s,t}$ is also a $\beta$-H\"older continuous multiplicative functional.

As an application of  Theorem \ref{thm_determ}  we deduce the convergence when the delay goes to zero  of the solutions for the stochastic differential delay equations
\begin{eqnarray}
X^r(t) &=& \eta(0) + \int_0^t b(u,X^r_u) \,du + \int_0^t \sigma(X^r_{u-r}) \,dB_u, \qquad t\in(0,T], \nonumber \\
X^r(t) &=& \eta(t), \qquad t\in[-r,0]. \nonumber
\end{eqnarray}
where the stochastic integral is a pathwise integral which depends on $B$ and $(B\otimes B)$. Set $X \equiv X^0$ the solution without delay and fix $\beta \in (\frac13,\frac12)$. Then the theorem states as follows:

\begin{thm} \label{thm_sto}
		Assume that $\sigma$ and $b$ satisfy {\upshape\bfseries (H1)} and {\upshape\bfseries (H2)} respectively, and both satisfy {\upshape\bfseries (H3)}. Assume also that $(\eta_{\cdot-r_0},B,\eta_{\cdot-r_0}\otimes B)\in M^\beta_{d,m}(0,r_0)$,  $\|\eta\|_{\beta(-r_0,0)}<\infty$ and $\sup_{r\leq r_0} \Phi_{\beta(0,r)}(\eta_{\cdot-r},B)<\infty$ a.s. Then,
			\[ \lim_{r\rightarrow0} \|X-X^r\|_{\infty} = 0 \quad a.s. \qquad {\rm and} \qquad 
			 \lim_{r\rightarrow0} \|(X\otimes B)-(X^r\otimes B)\|_{\infty} = 0 \quad a.s. \]
	\end{thm}

Applying   Theorem \ref{thm_determ} pathwise, the proof of Theorem  \ref{thm_sto} is an obvious consequence of (\ref{limBrBBr}) and (\ref{limBBrB}) of the following lemma.

\begin{lemma}
We have that
\begin{eqnarray}
\|B\otimes(B-B_{\cdot-r})\|_{2\beta'(r,T)} \longrightarrow 0 \qquad \text{a.s. when r tends to 0,}  \label{limBBBr} \\
\|B_{\cdot-r}\otimes(B-B_{\cdot-r})\|_{2\beta'(r,T)} \longrightarrow 0 \qquad \text{a.s. when r tends to 0,}  \label{limBrBBr} \\
\|(B-B_{\cdot-r})\otimes B\|_{2\beta'(r,T)} \longrightarrow 0 \qquad \text{a.s when r tends to 0.}  \label{limBBrB}
\end{eqnarray}
\end{lemma}

\begin{proof}	
Let us recall first that $ \|B\|_{\beta}<\infty$  a.s.
			
Then we begin estimating \eqref{limBBBr} when $i\neq j$ (we will consider the case $i=j$ at the end). By definition
\begin{eqnarray*}
&& \|B\otimes(B-B_{\cdot-r})\|_{2\beta'(r,T)} \\
&& \hspace{5mm} = \sup_{{s,t\in[r,T]}} \frac{1}{(t-s)^{2\beta'}} \bigg| \int_s^t (B^i_u-B^i_s) \,d^\circ B^j_{u-r} - \int_s^t (B^i_u-B^i_s) \,d^\circ B^j_u \bigg| 
\end{eqnarray*}
Assume first that $t-s>r$. Applying integration by parts, we have
	\begin{eqnarray} \label{A12}
		&& \|B\otimes(B-B_{\cdot-r})\|_{2\beta'(r,T)} \nonumber \\
		&& \hspace{5mm} = \sup_{\substack{s,t\in[r,T]\\ t-s>r}} \frac{1}{(t-s)^{2\beta'}} \bigg| (B^i_t-B^i_s)(B^j_{t-r}-B^j_{s-r}) - \int_s^t (B^j_{u-r}-B^j_{s-r}) \,d^\circ B^i_u \nonumber \\
		&& \hspace{45mm} + \int_s^t (B^j_u-B^j_s) \,d^\circ B^i_u - (B^i_t-B^i_s)(B^j_t-B^j_s) \bigg| \nonumber \\
		&& \hspace{5mm} \leq \sup_{\substack{s,t\in[r,T]\\ t-s>r}} \frac{1}{(t-s)^{2\beta'}} \bigg| (B^i_t-B^i_s)(B^j_{t-r}-B^j_t) \bigg| \nonumber \\
		&& \hspace{25mm} + \sup_{\substack{s,t\in[r,T]\\ t-s>r}} \frac{1}{(t-s)^{2\beta'}} \bigg| \int_s^t (B^j_u-B^j_{u-r}) \,d^\circ B^i_u \bigg| = A_1 + A_2.
	\end{eqnarray}
On one hand, by \eqref{yyr_inf}
\begin{eqnarray*}
A_1 &\leq& \sup_{\substack{s,t\in[r,T]\\ t-s>r}} \|B\|_{\beta'} \frac{|B^j_t-B^j_{t-r}|}{(t-s)^{\beta'}} \leq \sup_{\substack{s,t\in[r,T]\\ t-s>r}} \frac{1}{(t-s)^{\beta'}} \|B\|_{\beta'} \|B-B_{\cdot-r}\|_{\infty} \\
&\leq& \sup_{\substack{s,t\in[r,T]\\ t-s>r}} \frac{r^\beta}{(t-s)^{\beta'}} \|B\|_{\beta}^2 \,T^\varepsilon 
\leq \|B\|_{\beta}^2 \,T^\varepsilon \,r^\varepsilon
\end{eqnarray*}
that goes to zero when $r$ tends to zero.
On the other hand, we have that
$\int_s^t (B^j_u-B^j_{u-r}) \,d^\circ B^i_u $
is a continuous martingale, so it can be represented as a time-changed Brownian motion:
$ W_{\int_s^t (B^j_u-B^j_{u-r})^2 \,du}, $
where $W$ is a Brownian motion.
Now we choose $a\in(0,\frac12)$ such that $\frac{2\beta-2\varepsilon}{2\beta+1}<a<2\beta-2\varepsilon$. Applying H\"older property of the Brownian motion, we have
\begin{eqnarray*}
A_2 &=& \sup_{\substack{s,t\in[r,T]\\ t-s>r}} \frac{1}{(t-s)^{2\beta'}} \Big| W_{\int_s^t (B^j_u-B^j_{u-r})^2 \,du} \Big|\leq \sup_{\substack{s,t\in[r,T]\\ t-s>r}} \frac{C_{a,T}}{(t-s)^{2\beta'}} \Big| \int_s^t (B^j_u-B^j_{u-r})^2 \,du \Big|^a \\
%&\leq& \sup_{\substack{s,t\in[r,T]\\ t-s>r}} C_{a,T} \,\|B-B_{\cdot-r}\|_{\infty(r,T)}^{2a} \,\frac{(t-s)^{a}}{(t-s)^{2\beta'}} 
&\leq& \sup_{\substack{s,t\in[r,T]\\ t-s>r}} C_{a,T} \,\|B\|_{\beta}^{2a} \,r^{2a\beta} \,(t-s)^{a-2\beta'}
\leq  C_{a,T} \,\|B\|_{\beta}^{2a}\,r^{2a\beta+a-2\beta'}
\end{eqnarray*}
that clearly goes to zero when $r$ tends to zero thanks to the conditions on $a$.
		 
Now assume that $t-s\leq r$. By integration by part formula, we have
	\begin{eqnarray} 
		&& \|B\otimes(B-B_{\cdot-r})\|_{2\beta'(r,T)} \nonumber \\
%		&& \hspace{5mm} = \sup_{\substack{s,t\in[r,T]\\ t-s\leq r}} \frac{1}{(t-s)^{2\beta'}} \bigg| (B^i_t-B^i_s)(B^j_{t-r}-B^j_{s-r}) - \int_s^t (B^j_{u-r}-B^j_{s-r}) \,d^\circ B^i_u \nonumber \\
%		&& \hspace{60mm} - \int_s^t (B^i_u-B^i_s) \,d^\circ B^j_u \bigg| \nonumber \\
		&& \hspace{5mm} \leq \sup_{\substack{s,t\in[r,T]\\ t-s\leq r}} \frac{1}{(t-s)^{2\beta'}} \bigg| (B^i_t-B^i_s)(B^j_{t-r}-B^j_{s-r}) \bigg| \nonumber \\
		&& \hspace{15mm} + \sup_{\substack{s,t\in[r,T]\\ t-s\leq r}} \frac{1}{(t-s)^{2\beta'}} \bigg| \int_s^t (B^j_{u-r}-B^j_{s-r}) \,d^\circ B^i_u \bigg| \nonumber \\
		&& \hspace{15mm} + \sup_{\substack{s,t\in[r,T]\\ t-s\leq r}} \frac{1}{(t-s)^{2\beta'}} \bigg| \int_s^t (B^i_u-B^i_s) \,d^\circ B^j_u \bigg| = B_1 + B_2 + B_3. \label{B123}
	\end{eqnarray}
The first term is easy to bound, indeed,
\begin{eqnarray*}
B_1 = \sup_{\substack{s,t\in[r,T]\\ t-s\leq r}} \frac{|B^i_t-B^i_s|}{(t-s)^\beta} \cdot \frac{|B^j_{t-r}-B^j_{s-r}|}{(t-s)^\beta} \cdot \frac{(t-s)^{2\beta}}{(t-s)^{2\beta'}} \leq \|B\|_{\beta}^2 \,r^{2\varepsilon}.
\end{eqnarray*}
For the other two terms we use inequality (5.8) of Hu and Nualart \cite{HN}. It states that there exists a random variable $Z$ such that, almost surely, for all $s,t\in[0,T]$ we have
\[ \bigg| \int_s^t (B^i_u-B^i_s) \,d^\circ B^j_u \bigg| \leq Z |t-s| \log{\frac{1}{|t-s|}}. \]
Since the process $\{M'_t,t\in[s,T]\}$, defined as
$ M'_t = \int_s^t (B^j_{u-r}-B^j_{s-r}) \,d^\circ B^i_u, $
is  a continuous martingale, we can  follow the ideas in \cite{HN} to get the previous inequality in order to obtain that there exists a random variable $Z'$ such that, almost surely, for all $s,t\in[0,T]$ we have
\[ \bigg| \int_s^t (B^j_{u-r}-B^j_{s-r}) \,d^\circ B^i_u \bigg| \leq Z' |t-s| \log{\frac{1}{|t-s|}}. \]
Hence
\begin{eqnarray*}
B_2 \leq Z' (t-s)^{1-2\beta'} log{\frac{1}{(t-s)}} \leq Z' r^{1-2\beta'} log{\frac{1}{r}}
\end{eqnarray*}
and $B_2$ goes to zero when $r$ tends to zero. $B_3$ can be studied using the same arguments.
\vskip 3pt
\noindent
It only remains to prove the case where $i=j$. To simplify the notation we will not write the superindex $i$.
\vskip 3pt
\noindent
For $t-s\leq r$, we apply again the integration by parts formula and we obtain that
	\begin{equation*}
		\|B\otimes(B-B_{\cdot-r})\|_{2\beta'(r,T)} \leq B'_1 + B'_2 + B'_3 + B'_4,
	\end{equation*}
where $B'_1$, $B'_2$ and $B'_3$ are the terms defined in \eqref{B123} with $i=j$ and
	\begin{equation*}
		B'_4 := \sup_{\substack{s,t\in[r,T]\\ t-s\leq r}} \frac12 |t-s|^{1-\beta'} \leq \frac12 r^{1-\beta'}.
	\end{equation*}
So it only remains to study the terms $B'_1$, $B'_2$ and $B'_3$. 
Easily, for $B'_1$ we can repeat the same arguments used for $B_1$ and we also obtain that 
$B'_1\leq \|B\|_\beta^2r^{2\varepsilon}.$
If we focus in the second term, it can be written as
\[B'_2=\sup_{\substack{s,t\in[r,T]\\ t-s\leq r}}\frac{1}{(t-s)^{2\beta'}}\left|\int_s^t (B_{u-r}-B_{s-r})dB_u+\frac12\int_s^t D_u(B_{u-r}-B_{s-r})du\right|,\]
where $D_u$ denotes de Malliavin derivative. It is easy to check that this Malliavin derivative is zero. So, $B'_2$ it is now a martingale and we can proceed as in the case $i\neq j$, and we obtain 
\[B'_2\leq C r^{1-2\beta'}\log\frac{1}{r}.\]
Finally, for the last term we have
\[B'_3\leq \sup_{\substack{s,t\in[r,T]\\ t-s\leq r}}\frac{1}{2(t-s)^{2\beta'}}(B_t-B_s)^2\leq \frac12 \|B\|^2_\beta r^{2\varepsilon}.\]
Therefore, in that case also the three terms tend to zero when $r$ goes to zero.

\noindent
For the case $t-s>r$, by integration by parts formula we have
	\begin{eqnarray*}
		&& \|B\otimes(B-B_{\cdot-r})\|_{2\beta'(r,T)}  \leq \sup_{\substack{s,t\in[r,T]\\ t-s>r}} \frac{1}{(t-s)^{2\beta'}} \bigg| (B_t-B_s)(B_{t-r}-B_t) \bigg| \\
		&& \hspace{25mm} + \sup_{\substack{s,t\in[r,T]\\ t-s>r}} \frac{1}{(t-s)^{2\beta'}} \bigg| \int_s^t (B_u-B_{u-r}) \,d^\circ B_u - \frac12 (t-s) \bigg|.
	\end{eqnarray*}
The first term is analogous to the term $A_1$ defined in \eqref{A12}, so it is bounded by $\|B\|_{\beta}^2 \,T^\varepsilon \,r^\varepsilon$. \\
For the second term, we change from Stratonovich to It\^o integral. Since
	\[ \int_s^t (B_u-B_{u-r}) \,d^\circ B_u = \int_s^t (B_u-B_{u-r}) \,dB_u + \frac12 (t-s) .\]
For a fixed $s$, the process $\{M''_t,t\in[s,T]\}$, defined as
	$ M''_t = \int_s^t (B_u-B_{u-r}) \,dB_u$
is a continuous martingale. So following the ideas used for $A_2$, we obtain that
	\begin{eqnarray*}
		&& \sup_{\substack{s,t\in[r,T]\\ t-s>r}} \frac{1}{(t-s)^{2\beta'}} \bigg| \int_s^t (B_u-B_{u-r}) \,d^\circ B_u - \frac12 (t-s) \bigg| \\
		&& \hspace{25mm} \leq \sup_{\substack{s,t\in[r,T]\\ t-s>r}} \frac{1}{(t-s)^{2\beta'}} \bigg| \int_s^t (B_u-B_{u-r}) \,dB_u \bigg| \leq C_{a,T} \,\|B\|_{\beta}^{2a}\,r^{2a\beta+a-2\beta'},
	\end{eqnarray*}
where $a\in(0,\frac12)$ such that $\frac{2\beta-2\varepsilon}{2\beta+1}<a<2\beta-2\varepsilon$.

So finally, we obtain that $\|B\otimes (B-B_{\cdot-r})\|_{2\beta'(r,T)}\rightarrow 0$ as we wish.
\vskip 5pt

The inequality \eqref{limBrBBr} can be proved with similar computations and the proof of \eqref{limBBrB} follows immediately from the fact that
\[\|(B-B_{\cdot-r})\otimes B\|_{2\beta'(r,T)}\leq B_2+B_3.\]
\end{proof}

\end{document}